\documentclass[11pt]{amsart}
\usepackage{amssymb}
\usepackage{verbatim}
\usepackage{amsmath}
\usepackage{}
\usepackage{amssymb,latexsym}
\usepackage{enumerate}

\newcommand{\op} {\overline{\partial}}
\newcommand{\dbar}{\ensuremath{\overline\partial}}

\newcommand{\C}{\ensuremath{\mathbb{C}}}
\newcommand{\B}{\ensuremath{\mathbb{B}}}
\newcommand{\R}{\ensuremath{\mathbb{R}}}
\newcommand{\TL}{\textquotedblleft }
\newcommand{\TR}{\textquotedblright \ }

\makeatletter
\newcommand{\sumprime}{\if@display\sideset{}{'}\sum%
            \else\sum'\fi}
\makeatother

\begin{document}

\numberwithin{equation}{section}

\newtheorem{theorem}{Theorem}[section]
\newtheorem{proposition}[theorem]{Proposition}
\newtheorem{conjecture}[theorem]{Conjecture}
\def\theconjecture{\unskip}
\newtheorem{corollary}[theorem]{Corollary}
\newtheorem{lemma}[theorem]{Lemma}
\newtheorem{observation}[theorem]{Observation}
\newtheorem{definition}[theorem]{Definition}
\newtheorem{remark}{Remark}
\def\theremark{\unskip}
\newtheorem{kl}{Key Lemma}
\def\thekl{\unskip}
\newtheorem{question}{Question}
\def\thequestion{\unskip}
\newtheorem{example}{Example}
\def\theexample{\unskip}
\newtheorem{problem}{Problem}

\thanks{Research supported by Knut and Alice Wallenberg Foundation, and the China Postdoctoral Science Foundation.}

\address{School of Mathematical Sciences, Fudan University, Shanghai, 200433, China}
\address{Current address: Department of Mathematical Sciences, Chalmers University of Technology and
University of Gothenburg. SE-412 96 Gothenburg, Sweden}
\email{wangxu1113@gmail.com}
\email{xuwa@chalmers.se}

\title[Curvature Formula]{A curvature formula associated to a family of pseudoconvex domains}
 \author{Xu Wang}
\date{}
\maketitle

\begin{abstract} We shall give a definition of the curvature operator for a family of weighted Bergman spaces $\{\mathcal H_t\}$ associated to a smooth family of smoothly bounded strongly pseudoconvex domains $\{D_t\}$. In order to study the \TL boundary term\TR  in the curvature operator, we shall introduce the notion of geodesic curvature for the associated family of boundaries $\{\partial D_t\}$. As an application, we get a variation formula for the norms of Bergman projections of currents with compact support. A flatness criterion for $\{\mathcal H_t\}$ and its applications to triviality of fibrations are also given in this paper.
\bigskip


\smallskip

\end{abstract}

\tableofcontents

\section{Introduction}

In 2005, Berndtsson \cite{Bern06} found that the functional version of the classical Brunn-Minkowski inequality, i.e. the Prekopa theorem \cite{Prekopa73}, can be seen as a special case of the subharmonicity property of the Bergman kernel (see \cite{Maitani84} and \cite{MY04} for early results with different point of view). It opens another door (so called \emph{complex Brunn-Minkowski theory}) of studying complex geometry by using the Brunn-Minkowski theory in convex geometry. Below, we shall give a short account of the complex Brunn-Minkowski theory, and a simple example to show our motivation to write this paper.

Heuristically speaking, the Prekopa theorem can be seen as a version of inverse H\"older inequality. In \cite{Bern13a}, Berndtsson gave another form of the classical H\"older inequality:

\medskip

\textbf{Theorem A} [H\"older inequality]: Let $\phi(t, x)$ be convex in $t$. Then
\begin{equation}
t \mapsto \log\int_{\mathbb R^n} e^{\phi(t,x)} dx, \ dx=dx^1\wedge \cdots \wedge dx^n,
\end{equation}
is convex (if the integral is convergent).

\medskip

The proof follows by differentiating with respect to $t$: 
\begin{equation}
\left(\log\int e^{\phi}  \right)_{tt}=\left(\int  e^{\phi} \right)^{-2}  \left(\int e^{\phi}  \int (\phi_{tt}e^{\phi} +\phi_t^2 e^\phi)-\big(\int \phi_t e^\phi\big)^2 \right).
\end{equation}
Notice that, by the Cauchy-Schwarz inequality, we have
\begin{equation}
\big(\int \phi_t e^\phi\big)^2 \leq \int e^{\phi}  \int \phi_t^2 e^\phi.
\end{equation}
Thus $\phi_{tt} \geq 0$ implies that $\left(\log\int e^{\phi}  \right)_{tt}\geq 0$. 

\medskip

The following result is due to Prekopa:

\medskip

\textbf{Theorem B} [Prekopa theorem]: Let $\phi(t, x)$ be convex in $t$ and $x$. Then
\begin{equation}
t \mapsto -\log\int_{\mathbb R^n} e^{-\phi(t,x)} dx,
\end{equation}
is convex.

\medskip

There are many ways to prove Theorem B. A famous observation of Brascamp-Lieb (see \cite{BL76}) is:  one may use a weighted $L^2$-estimates of the $d$-operator to prove Theorem B. 

In \cite{Bern98}, Berndtsson showed that one may also use the H\"ormander's weighted $L^2$-estimates of the $\dbar$-operator (see \cite{Hormander65}) to prove Theorem B. Moreover, in \cite{Bern06}, he established the following complex version of Theorem B:

\medskip

\textbf{Theorem C} [Berndtsson's theorem]: Let $\phi(t, z)$ be a plurisubharmonic function on a pseudoconvex domain  $D \subset \C^m_t\times \C^n_z$. Then
\begin{equation}
(t,z) \mapsto \log K^t(z,z),
\end{equation}
is plurisubharmonic or equal to $-\infty$ identically on $D$, where each $K^t$ denotes the weighted Bergman kernel associated to the fibre $D_t:=D\cap( \{t\}\times \C^n)$ and the weight $\phi^t:=\phi|_{D_t}$.

\medskip

In \cite{Bern06}, Berndtsson gave two proofs of Theorem C. A crucial step in his first proof is also the H\"ormander's $L^2$-estimates of the $\dbar$-operator. Later in \cite{Bern09}, he pointed out that it will be more natural to look at Theorem C as a curvature property of the direct image bundle (see Theorem 1.1 and Theorem 1.2 in \cite{Bern09}). This is a milestone in the complex Brunn-Minkowski theory (see \cite{Bern13b}).  

The complex Brunn-Minkowski theory has proved to be very useful in several complex variables and complex geometry (see \cite{Bern09a}, \cite{Bern13}, \cite{Bern14}, \cite{BL14}, \cite{BB14}, \cite{BernPaun08} and references therein). This paper is an attempt to study the curvature formula of the direct image bundle associated to general Stein-fibrations (see \cite{Tsuji05}, \cite{Sch12}, \cite{LiuYang13}, \cite{GS15}, \cite{MT07} and \cite{MT08} for other generalizations and related results). The new results are the \emph{boundary term} of the curvature formula and its relation with \emph{interpolation family} of convex bodies.  

Let us start by looking at an almost trivial case of Theorem B. Let 
\begin{eqnarray}
\mathcal F:= \{[a(t),b(t)]\}_{0\leq t\leq 1},
\end{eqnarray}
be a family of line segments. Let
\begin{equation*}
    D:=\{(t,x)\in \mathbb R^2: a(t)<x<b(t), \ 0<t<1\},
\end{equation*}
be the total space. Assume that $b(t)> a(t)$ for each $0\leq t\leq 1$ and $a, b$ are smooth on a neighborhood of $[0,1]$. Put
\begin{equation*}
    \theta(a)=\frac{d^2a}{dt^2},\ \theta(b)=-\frac{d^2b}{dt^2}.
\end{equation*}
Let us introduce the following definitions:

\begin{definition} We call $\theta$ the geodesic curvature of $\mathcal F$.  
\end{definition}

\begin{definition} 
We call $\mathcal F$ an interpolation family if $\theta\equiv 0$.
\end{definition}

\textbf{Remark:}  $\mathcal F$ is an interpolation family if and only if both $a$ and $b$ are affine functions. 

\begin{definition}  We call $\mathcal F$ a trivial family if there exists a real constant $c$ such that for every $0< t<1$, $[a(t),b(t)]=[a(0),b(0)]+ct$. 
\end{definition}

Put 
\begin{equation}
\phi(t,x)=0, \ \text{on} \ D, \  \ \phi(t,x)=\infty,\ \text{on} \ \mathbb R^2 \backslash D,
\end{equation}
then convexity of $D$ is equivalent to convexity of $\phi$.  Thus Theorem B implies that if $D$ is convex then
\begin{equation}
\Phi:  t \mapsto -\log (b(t)-a(t)) =-\log\int_{\mathbb R} e^{-\phi(t,x)} dx 
\end{equation}
is convex on $(0,1)$. Moreover, by direct computation,
\begin{equation}
\ddot \Phi=\frac{(b-a)(\ddot a-\ddot b)+(\dot a-\dot b)^2}{(b-a)^2},   \ \ddot \Phi:=\frac{d^2\Phi}{dt^2}, \ \dot a:=\frac{da}{dt}.
\end{equation}
We call 
\begin{equation}
Geo:=\frac{(b-a)(\ddot a-\ddot b)}{(b-a)^2}=\frac{\theta(a)+\theta(b)}{b-a},
\end{equation}
the geodesic term in $\ddot \Phi$ and
\begin{equation}
R:=\frac{(\dot a-\dot b)^2}{(b-a)^2}
\end{equation}
the remaining term in $\ddot \Phi$. Thus we have:

\begin{proposition}\label{pr:curvature-real} The remaining term in $\ddot\Phi$ is always non-negative. Moreover, if the total space $D$ is convex then the geodesic term in $\ddot\Phi$ is also non-negative.
\end{proposition}

\begin{proposition}\label{pr:trivial-real} Assume that the total space $D$ is convex then affine-ness of $\Phi$ is equivalent to triviality of $\mathcal F$. 
\end{proposition}

In this paper, we shall study the counterparts of the above notions in complex geometry. In the next secion, we shall define the notion of geodesic curvature (see Definition \ref{de:geodesic-c-bdy}) for a smooth family of smoothly bounded Stein domains (see Definition \ref{de:smoothfamily}). Then Definition \ref{de:inter-psc}, \emph{interpolation family of Stein domains}, can be seen as a generalization of Definition 1.2; and Definition \ref{de:trivial}, \emph{trivial family of Stein domains}, can be seen as a generalization of Definition 1.3.

Our main result, Theorem \ref{th:CF}, is a curvature formula associated to variation of Stein manifolds. Let $\{D_t\}$ be a smooth family of smoothly bounded $n$-dimensional Stein domains. Let $\mathcal L$ be a holomorphic line bundle on the total space $D$. Let $h$ be a smooth Hermitian metric on $\mathcal L$. We shall consider the associated family of Bergman spaces
\begin{equation*}
\mathcal H:=\{\mathcal H_t\},
\end{equation*}
where each $\mathcal H_t$ is the space of $L^2$-holomorphic $\mathcal L|_{D_t}$-valued $(n,0)$-forms on $D_t$. Then Theorem \ref{th:CF} reads that:
 
\medskip

\textbf{Main Theorem}: \emph{Assume that $\mathcal L$ is flat or relatively ample. Then the curvature of $\mathcal H$ contains two terms: the geodeisc term and the remaining term. The remaining term is always semi-positive in the sense of Nakano. Moreover, if the total space is pseudoconvex and $\mathcal L$ is semi-positive on the total space then the geodesic term is also semi-positive in the sense of Nakano.}

\medskip

Thus our main result can be seen as a generalization of Proposition \ref{pr:curvature-real}. In section 2.5, we shall give a definition of the holomorphic section of the \textbf{dual} of $\mathcal H$ (see Definition \ref{de:holomorphic-dual-new}).  
Then our main application, Corollary \ref{co:dual-psh}, can be stated as follows:

\medskip

\textbf{Application}: \emph{If the total space is pseudoconvex and $\mathcal L$ is non-negative on the total space then $\log||f||$ is plurisubharmonic for every holomorphic section $f$ of the dual of $\mathcal H$.}

\medskip

Corollary \ref{co:dual-psh} can be seen as a generalization of Theorem C (see  the remark behind Theorem 1.1 in \cite{Bern09}). In section 5.2, we shall also use Corollary \ref{co:dual-psh} to study variation of the Bergman projection of currents with compact support. In particular, we shall give a variation formula for the derivatives of the Bergman kernel (see Theorem \ref{th:VF}).

In section 6, we shall discuss the counterparts of Proposition \ref{pr:trivial-real} in complex case. We shall show that under some assumptions (see Theorem \ref{th:flat}), flatness of $\mathcal H$ and triviality of $D$ are equivalent. As a direct corollary, we shall give a triviality criterion for a class of holomorphic motions (see Corollary \ref{co:last}) of planar domains. 

\medskip

\textbf{Acknowledgement}: I would like to thank Bo-Yong Chen for introducing me this topic, Bo Berndtsson for many inspiring discussions relating the curvature of the direct images, his  useful comments on this paper and suggestions on paper writing. Thanks are also given to Qing-Chun Ji for his constant support and encouragement.

I would also like to thank Laszlo Lempert for pointing out a mistake in the first version of this paper regarding the assumptions of Lemma \ref{le:key-lemma}. After the first version of this manuscript was completed, Laszlo Lempert kindly sent me a preprint \cite{Tran} of Dat Vu Tran. In \cite{Tran}, Dat Vu Tran gave a careful study of curvature of Hilbert fields (see also \cite{LS14}) and showed that the curvature operator associated to a famlily of planar domains with pseudoconvex global space is semi-positive, which covers our main theorem in case each fibre is one-dimensional.

Last but not least, thanks are due to the referee for pointing out several inaccuracies in this paper and hopefully making this paper more readable.

\section{Basic definitions and results}

\subsection{List of notations} \ \

\medskip

1. $\pi: \mathcal X \to \mathbb B$ is a holomorphic submersion.

2. $D$ is an open subset in $\mathcal X$, $D_t:=D\cap \pi^{-1}(t).$

3. $\mathcal L$ is a holomorphic line bundle over $\mathcal X$ and $L_t:=\mathcal L|_{D_t}$.

4. $\mathcal H_t$ is the space of $L^2$ holomorphic $L_t$-valued $(n,0)$-forms on $D_t$.

5. $\mathcal H:=\{\mathcal H_t\}_{t\in \mathbb B}$.

6. $i_t$: the inclusion mapping $D_t\hookrightarrow D$.

7. $t$: coordinate system on $\mathbb B$, $t^j$ components of $t$, $\partial_{t^j}=\partial/\partial t^j$, $\dbar_{t^k}=\partial/\partial \bar t^k$.

8. $\sum d\bar t^j\otimes \dbar_{t^j}$ is the $\dbar$-operator on $\mathcal H$.

9. $\sum dt^j\otimes D_{t^j}$ is the $(1,0)$-part of the Chern connection on $\mathcal H$.  

10. $\Theta_{j\bar k}:=[D_{t^j}, \dbar_{t^k}]$: curvature operators on $\mathcal H$.

11. $\theta_{j\bar k}(\rho)$: geodesic curvature of $\{\partial D_t\}$.

12. $\eta, \zeta, \mu$: local coordinate system on a fixed fibre $D_t$ of $D$, $\mu^\lambda$: components of $\mu$.

13. $e^{-\phi}$: local representative of a Hermitian metric $h$ on a line bundle $\mathcal L$.

14. $\phi_j:=\partial\phi/\partial t^j$, $\phi_{j\lambda}:=\partial^2\phi/\partial t^j\partial \mu^{\lambda}$, $\phi_{\lambda\bar\nu}:=\partial^2\phi/\partial \mu^\lambda\partial \bar\mu^{\nu}$.

15. $(\rho^{\bar\lambda\nu})$, $(\phi^{\bar\lambda\nu})$: inverse matrix of $(\rho_{\lambda\bar\nu})$, $(\phi_{\lambda\bar\nu})$ respectively.

16. $\delta_{V}:=V~\lrcorner ~$ means contraction of a form with a vector field $V$;

17. $\alpha,\beta\in \mathbb N^n$, $|\alpha|:=\alpha_1+\cdots+\alpha_n$, $f_\alpha:=\partial^{|\alpha|}f/(\partial\mu^1)^{\alpha_1}\cdots
(\partial\mu^n)^{\alpha_n}$.

\subsection{Set up}

Let $\pi:\mathcal X\to \mathbb B$ be a holomorphic submersion from an $(n+m)$-dimensional complex manifold $\mathcal X$ to the unit ball $\mathbb B$ in $\mathbb C^m$. Let $D$ be an open subset of $\mathcal X$. Put
\begin{equation*}
    D_t=D\cap \pi^{-1}(t).
\end{equation*}
The following assumption will be used throughout this paper.

\medskip

\textbf{A1:}  \emph{The restriction of $\pi$ to the closure of D (with respect to the topology structure of the total space $\mathcal X$) is \textbf{proper},  and $D_t$ is non-empty for every $t$ in $\mathbb B$.}

\medskip

Let $\mathcal L$ be a holomorphic line bundle over $\mathcal X$. Put
\begin{equation*}
 L_t:=\mathcal L|_{D_t}.
\end{equation*}
We shall consider the following family of vector spaces associated to $\{D_t\}_{t\in\mathbb B}$:
\begin{equation*}
    \mathcal H_t:=\{f\in H^0(D_t, K_{D_t}+L_t): \int_{D_t}i^{n^2}\{f,f\} <\infty \},
\end{equation*}
where $\{\cdot, \cdot\}$ is the canonical sesquilinear pairing (see page 268 in \cite{Demailly12}) with respect to $h$. If we fix a local holomorphic frame, say $e$, of $\mathcal L$, and write $h(e,e)=e^{-\phi}$, then
\begin{equation*}
\{f(z)\otimes e,f(z)\otimes e\}=e^{-\phi(z)} f(z)\wedge \overline{f(z)}.
\end{equation*}
Put
\begin{equation*}
    \mathcal H=\{\mathcal H_t\}_{t\in\mathbb B}.
\end{equation*}

\medskip

\textbf{Remark:} We know that each fibre $\mathcal H_t$ is an infinite dimensional Hilbert space, moreover, the element in $\mathcal H_t$ may not be smooth up to the boundary. Thus it is not easy to give a good definition of the curvature operator on $\mathcal H$ (see \cite{LS14} for a careful study of this subject). In order to make things easier,  \emph{we shall only define the curvature operator on sections that are smooth up to the boundary.} 

\medskip

Let $i_t$ be the inclusion mapping
\begin{equation}
i_t: D_t\hookrightarrow D.
\end{equation}
We shall introduce the following definition:

\begin{definition}\label{de:smooth} We call $u: t\mapsto u^t \in \mathcal H_t$
a smooth section of $\mathcal H$ if there exists an $\mathcal L$-valued $(n,0)$-form, say $\mathbf u$, such that 
\begin{equation}
i_t^* \mathbf u=u^t, \ \forall\ t\in\mathbb B,
\end{equation}
and $\mathbf u$ is \textbf{smooth up to the boundary} of $D$. We shall denote by $\Gamma(\mathcal H)$ the space of smooth sections of $\mathcal H$. 
\end{definition}

\textbf{Remark}: One may choose $\mathbf u$ in the above definition such that $\mathbf u$ is smooth \textbf{on the total space $\mathcal X$}. 

\medskip

In order to give a precise definition of the $\dbar$-operator on $\mathcal H$, we shall introduce the following definition (see \cite{Bern11}, see also Page 46 in \cite{KSp} or \cite{Wang13} for the admissible coordinate method):

\begin{definition}\label{de:re} We call a smooth $\mathcal L$-valued $(n,0)$-form $\textbf{u}$ on $\mathcal X$ a \textbf{representative} of $u\in\Gamma (\mathcal H)$ if $i_t^* (\textbf{u})=u^t$ for all $t\in\mathbb B$.
\end{definition}

\medskip

\textbf{$\dbar$-operator on $\mathcal H$:} Let $\mathbf u$ be a representative of $u\in\Gamma(\mathcal H)$. Then one may write
\begin{equation}
\dbar \mathbf u=\sum dt^j\wedge \eta_j +d\bar t^j\wedge \nu_j. 
\end{equation}
Since $\mathbf u\wedge dt$, $dt:=dt^1\wedge \cdots \wedge dt^m$, does not depend on the choice of $\mathbf u$ and
\begin{equation}
\dbar (\mathbf u\wedge dt)= \sum d\bar t^j\wedge \nu_j \wedge dt,
\end{equation}
we know that each $i_t^*\nu_j$ does not depend on the choice of $\nu_j$. Let us define
\begin{equation}\label{eq:def-dbar25}
\dbar_{t^j} u: t\mapsto  i_t^*\nu_j.
\end{equation}
Then the $\dbar$-operator on $\mathcal H$ can be defined as
\begin{equation*}
\dbar u:=\sum d\bar t^j\otimes \dbar_{t^j} u.
\end{equation*} 
From the definition, we know that $\dbar u\equiv 0$ on $\mathbb B$ if and only if   $\mathbf u \wedge dt$ is holomorphic on $D$. We shall introduce the following definition:

\begin{definition}\label{de: holo-section-H} Let $u$ be a smooth section of $\mathcal H$. We call $u$ a holomorphic section of $\mathcal H$ if $\mathbf u\wedge dt$ is holomorphic on $D$.
\end{definition}

\textbf{Chern connection on $\mathcal H$:} We shall write the $(1,0)$-part of the Chern connection on $\mathcal H$ as $\sum dt^j\otimes D_{t^j}$. By definition, each $D_{t^j}$ should satisfy
\begin{equation}\label{eq:chern-connection}
    \partial_{t^j}(u,v)=(D_{t^j} u, v) + (u, \dbar_{t^j} v), \ \forall\ u, v\in \Gamma(\mathcal H),
\end{equation}
where $\partial_{t^j}=\partial/\partial t^j$ and $(\cdot,\cdot)$ denotes the inner product on $\mathcal H_t$.

\begin{definition}\label{de:cc} We say that the Chern connection is well defined on $\mathcal H$ if for every $1\leq j\leq m$, there exists a $\C$-linear operator $D_{t^j} : \Gamma (\mathcal H) \to \Gamma (\mathcal H)$ such that \eqref{eq:chern-connection} is true.
\end{definition}

\textbf{Remark:} By using Hamilton's theorem (see \cite{Hamilton79}), in section 4 (see Proposition \ref{pr:chern}), we shall prove that the Chern connection is well defined on $\mathcal H$ if $D$ satisfies $\mathbf{A1}$ and the following assumption:

\medskip

\textbf{A2:}  \emph{There is a smooth real valued function $\rho$ on $\mathcal X$ such that for each $t\in\mathbb B$, $\rho|_{\pi^{-1}(t)}$ is strictly plurisubharmonic in a neighborhood of the closure $($with respect to the topology on $\pi^{-1}(t))$ of $D_t$. Moreover, $D_t=\{\rho <0\}\cap\pi^{-1}(t)$ and the gradient of  $\rho|_{\pi^{-1}(t)}$ has no zero point on $\partial D_t$. }

\medskip

\textbf{Remark:} In section 3.2, we shall prove that $\mathbf A1$ and $\mathbf A2$ together implies that every smooth vector field on the base $\mathbb B$ has a smooth lift on $\mathcal X$ that tangent to the boundary of $D$. Thus in this case, $\{D_t\}$ is locally trivial as a smooth family.

\begin{definition}\label{de:smoothfamily} $\{D_t\}$ is said to be a smooth family of smoothly bounded Stein domains if $D$ satisfies $\mathbf{A1}$ and $\mathbf{A2}$.
\end{definition}

Assume that $D$ satisfies $\mathbf{A1}$ and $\mathbf{A2}$. Then we can define the curvature operators on $\mathcal H$ as follows:
\begin{equation}\label{eq:new-1}
    \Theta_{j\bar k}u:=[D_{t^j}, \dbar_{t^k}]u=D_{t^j}\dbar_{t^k}u-\dbar_{t^k} D_{t^j}u, \ \forall \ u\in \Gamma(\mathcal H).
\end{equation}

\subsection{Previous results}

We shall give a short account of Berndtsson's results on\ \textquotedblleft geodesic\textquotedblright\ formula for $\Theta_{j\bar k}$. Let us recall the following notions in his curvature formula.  

\medskip

\textbf{Geodesic curvature in the space of K\"ahler metrics}: Let us denote by $\Theta(\mathcal L,h)$ the curvature of $(\mathcal L,h)$. If we write $h$ locally as $e^{-\phi}$ then we have
\begin{equation}
\Theta(\mathcal L,h) = \partial\dbar \phi.
\end{equation} 
If $D$ is a product, say $D=D_0 \times\mathbb B$, and 
\begin{equation}
i\partial\dbar \phi|_{D_0\times \{t\}}  >0, \ \forall\ t\in\mathbb B,
\end{equation}
then $\{i\partial\dbar \phi|_{D_0\times \{t\}} \}$ can be seen as a family of K\"ahler metrics on $D_0$. Assume further that $m=1$. Then there exists a smooth function, say $c(\phi)$, such that
\begin{equation}
\frac{(i\partial\dbar \phi)^{n+1}}{(n+1)!}= c(\phi)\frac{(i\partial\dbar \phi)^{n}}{n!} \wedge idt\wedge d\bar t.
\end{equation}
By Proposition 3 in \cite{Donaldson99}, if $\{i\partial\dbar \phi|_{D_0\times \{t\}} \}$ is $S^1$ invariant then

\medskip

\emph{The path $\{i\partial\dbar \phi|_{D_0\times \{t\}} \}$ defines a geodesic in the space of K\"ahler metrics on $D_0$ if and only if $c(\phi)\equiv 0$.}

\medskip

In general, $c(\phi)$ is called the \emph{geodesic curvature} in the space of K\"ahler metrics. The geodesic curvature plays a crucial role on variation of K\"ahler metrics on projective manifolds; see \cite{Mabuchi86}, \cite{Semmes88} and \cite{Donaldson99}, to cite just a few.  Another way to look at the geodesic curvature is to use the notion of \emph{horizontal lift}.

\medskip

\textbf{Horizontal lift}: The notion of horizontal lift is introduced by Siu in \cite{Siu86}. In \cite{Bern11}, Berndtsson found that one may also define the notion of horizontal lift with respect to a relative K\"ahler form (a smooth $d$-closed $(1,1)$-form that is positive on each fibre). Let us recall his definition:

\medskip

\emph{Let $\omega$ be a relative K\"ahler form on $\mathcal X$. A $(1,0)$-vector field $V$ on $\mathcal X$ is said to be  horizontal with respect to $\omega$ if  $\langle V, W\rangle_{\omega}=0$ for every $(1,0)$-vector field $W$ such that $\pi_*(W)=0$. Let $v$ be a vector field on $\mathbb B$. We call $V$ a \emph{horizontal lift} of $v$ if $V$ is horizontal with respect to $\omega$ and $\pi_*V$=$v$.}

\medskip

By Berndtsson's formula (see page 3 in \cite{Bern11}), if we write $\omega=i\partial\dbar\phi$ locally then for each $1\leq j\leq m$, $\partial/\partial t^j$ has a unique horizontal lift, say $V_j$, as follows:
\begin{equation}
V_j=\partial/\partial t^j-\sum \phi_{j\bar \nu}\phi^{\bar \nu \lambda} \partial/\partial \mu^\lambda.
\end{equation} 
Moreover, if $m=1$ then
\begin{equation}
\langle V_1, V_1\rangle_{i\partial\dbar\phi}=c(\phi).
\end{equation}
By this formula, it is natural to define the notion of geodesic curvature for general base dimension $m$ and general fibration $\pi$. 

\medskip

\textbf{Geodesic curvature of $\{h|_{L_t}\}$}: Let us assume that
\begin{equation*}
    i\Theta(\mathcal L,h)|_{D_t}>0, \ \forall \ t\in\mathbb B, \ \text{or} \ \Theta(\mathcal L,h)\equiv 0 \ \text{on}\ D,
\end{equation*}
In case $i\Theta(\mathcal L,h)|_{D_t}> 0, \ \forall \ t\in\mathbb B$,  let $V_j^h$ be the horizontal lift of the base vector fields $\partial/\partial t^j$ with respect to $i\Theta(\mathcal L,h)$. We shall define the geodesic curvature of $\{h|_{L_t}\}$ as:
\begin{equation}\label{eq:cjkh}
    c_{j\bar k}(h):=\langle V_j^h, V_k^h\rangle_{i\Theta(\mathcal L,h)}; \ \text{and} \ c_{j\bar k}(h):=0 \ \text{if} \ \Theta(\mathcal L,h)\equiv 0 \ \text{on} \ D. 
\end{equation}
Another notion in Berndtsson's curvature formula is the following:

\medskip

\textbf{Remaining term in H\"ormander's $L^2$-estimate (product case)}: Assume that $D=D_0\times \mathbb B$. Then the vector fields $\partial/\partial t^j$ are well defined on $\mathcal X$. Fix  \begin{equation}\label{eq:new-2}
       u_j\in\Gamma(\mathcal H), \ 1\leq j\leq m, \ \ \text{(see Definition \ref{de:smooth})}.
\end{equation}
Let $a$ be the $L^2$-minimal solution of 
\begin{equation*}
    \dbar^t(\cdot)=c,
\end{equation*}
where $\dbar^t$ denotes the Cauchy-Riemann operator on $D_t=D_0\times \{t\}$ and
\begin{equation}
c:=\sum ( \partial/\partial t^j \lrcorner ~\Theta(\mathcal L,h))|_{D_t} \wedge u_j=\sum\dbar^t\phi_{j}\wedge u_j.
\end{equation} 
Then we have the following remaining term in H\"ormander's $L^2$-estimate
\begin{equation*}
    R:=||c||^2_{i\Theta(\mathcal L,h)|_{D_t}} -||a||^2.
\end{equation*}
By H\"ormander's theorem (see \cite{Hormander65}), if $D_0$ is pseudoconvex then $R$ is non-negative. Thus in our case, $R$ is always non-negative. Now we can state the following theorem of Berndtsson (see \cite{Bern06} and Theorem 1.1 in \cite{Bern09}):

\begin{theorem}\label{th:CF-06} Assume that $D=D_0\times \mathbb B$, where $D_0$ is a strongly pseudoconvex domain with smooth boundary. If $i\Theta(\mathcal L,h)|_{D_0\times \{t\}}>0, \ \forall \ t\in\mathbb B$, then we have
\begin{equation}\label{eq:final-CF-06}
    \sum(\Theta_{j\bar k} u_j, u_k)= \sum(c_{j\bar k}(h) u_j, u_k) + R,  \ R\geq 0. 
\end{equation} 
If $\Theta(\mathcal L,h)\equiv 0 \ \text{on}\ D$ then
\begin{equation}\label{eq:final-CF-06-1}
    \sum(\Theta_{j\bar k} u_j, u_k)\equiv 0.
\end{equation}
\end{theorem}

\textbf{Remark}: \eqref{eq:final-CF-06} can be found in the proof of Theorem 1.1 in \cite{Bern09}. \eqref{eq:final-CF-06-1} is a direct application of formula $(2.4)$ in \cite{Bern09}. A special case of \eqref{eq:final-CF-06} for variation of the Bergman kernel is given in \cite{Bern06}.

\medskip

The counterpart of Theorem \ref{th:CF-06} for a proper K\"ahler fibration was given in \cite{Bern11} by Berndtsson. Let us recall the following notions in his formulae:

\medskip

\textbf{Remaining term in H\"ormander's $L^2$-estimate (Polarized fibration)}: Let 
\begin{equation}
\pi: D \to \mathbb B,
\end{equation}
be a proper holomorphic submersion. Assume that $\mathcal L$ is a relatively ample line bundle on $D$, i.e. $i\Theta(\mathcal L,h)|_{D_t}>0, \ \forall \ t\in\mathbb B$. By the Ohsawa-Takegoshi extension theorem (see \cite{OT87} and \cite{Siu02}), we know that the dimension of $\mathcal H_t:=H^0(D_t, K_{D_t}+L_t)$ does not depend on $t$ and our bundle $\mathcal H$ is just the holomorphic vector bundle associated to the zero-th direct image sheaf $\pi_*\mathcal O(K_{D/\mathbb B}+\mathcal L)$. For each $j$, let $V_j^h$ be the horizontal lift of $\partial/\partial t^j$ with respect to $i\Theta(\mathcal L,h)$. Let us denote by
\begin{equation}
\kappa: T_{\mathbb B} \to \{H^1(D_t, T_{D_t})\}_{t\in\mathbb B},
\end{equation}
the Kodaira-Spencer map associated to the holomorphic fibration $\pi$. By Theorem 5.4 in \cite{Kodaira86}, we know that each $(\dbar V_j^h)|_{D_t}$ can be seen as a representative of the Kodaira-Spencer class $\kappa(\partial/\partial t^j)$. For each $1\leq j\leq m$, let  $u_j$ be a smooth section of $\mathcal H$. Put
\begin{equation*}
    b=\sum (\dbar V_j^h)|_{D_t}\lrcorner ~u_j. 
\end{equation*}
Let $a$ be the $L^2$-minimal solution of 
\begin{equation*}
    \dbar^t(\cdot)=\partial^t_\phi b,
\end{equation*}
where $\partial^t_\phi $ is the restriction of the $(1,0)$-part of the Chern connection of $\mathcal L$ on $D_t$. Then we have the following H\"ormander type-remaining term:
\begin{equation*}
    R^h:=||b||^2_{i\Theta(\mathcal L,h)|_{D_t}} -||a||^2\geq 0.
\end{equation*}

\medskip

\textbf{Remaining term in H\"ormander's $L^2$-estimate (K\"ahler fibration)}: In this case, let us assume that the total space of the proper holomorphic submersion $\pi: D \to \mathbb B$ possesses a K\"ahler form $\omega$. Let $\mathcal L$ be a flat line bundle over $D$, i.e.
$\Theta(\mathcal L,h)\equiv 0 \ \text{on}\ D$. By Theorem 8.1 in \cite{Bern09}, we know that $\pi_*\mathcal O(K_{D/\mathbb B}+\mathcal L)$ is locally free. Let $\mathcal H$ be the associated vector bundle. For each $j$, let $u_j$ be a smooth section of $\mathcal H$. Put
\begin{equation*}
    b=\sum (\dbar V_j^\omega)|_{D_t}\lrcorner ~u_j,
\end{equation*}
where each $V_j^{\omega}$ is the horizontal lift of $\partial/\partial t^j$ with respect to the K\"ahler form $\omega$. Consider
\begin{equation*}
    \dbar^t(a)=\partial^t_\phi b,
\end{equation*}
where $a$ is the $L^2$-minimal solution. Then the associated H\"ormander remaining term is
\begin{equation*}
    R^{\omega}:=||b||^2_{\omega|_{D_t}} -||a||^2\geq 0.
\end{equation*}
Now we can state the following theorem of Berndtsson (see \cite{Bern11}):

\begin{theorem}\label{th:CF-11} Let $\pi: D \to \mathbb B$ be a proper holomorphic submersion. Let $(\mathcal L,h)$ be a holomorphic line bundle over $D$. If $\mathcal L$ is  relatively ample then we have
\begin{equation}\label{eq:final-CF-11}
    \sum(\Theta_{j\bar k} u_j, u_k)= \sum(c_{j\bar k}(h) u_j, u_k) + R^h. 
\end{equation} 
If $\mathcal L$ is flat then
\begin{equation}\label{eq:final-CF-11-1}
    \sum(\Theta_{j\bar k} u_j, u_k)\equiv R^{\omega}.
\end{equation}
\end{theorem}

\textbf{Remark}: If $\mathcal L$ is trivial then \eqref{eq:final-CF-11-1} is just Griffiths' formula (see page 33 in \cite{Griffiths84}). In general, if the total space $D$ is K\"ahler and there is a smooth Hermitian metric on $\mathcal L$ with non-negative curvature then by Thereom 1.2 in \cite{Bern09}, we know that 
\begin{equation*}
    \sum(\Theta_{j\bar k} u_j, u_k)\geq 0.
\end{equation*}

\medskip

If the boundary of each fibre is non-empty then, in general, the boundary term should also appear in the curvature formula (see \cite{MY04}). Our main result is a study of the curvature of $\mathcal H$ for fibrations with boundary.  

\subsection{Basic notions for fibrations with boundary} Let $D=\{D_t\}_{t\in\mathbb B}$ be a smooth family of smoothly bounded Stein domains (see Definition \ref{de:smoothfamily}). We shall define the notion of \textquotedblleft geodesic curvature\textquotedblright\ of $\{\partial D_t\}$ by using the notion of  \emph{horizontal lift with respect to the Levi-form on the boundary} of $D$. Let $\rho$ be the defining function in $\mathbf{A2}$. We call an $(1,0)$-\emph{tangent} vector field $V$ on $\partial D$ horizonal with respect to the Levi-form if
\begin{equation*}
\langle V, W \rangle_{i\partial\dbar\rho}=0, \ \text{on}\ \partial D,
\end{equation*}
for every $(1,0)$-\emph{tangent} vector field $W$ on $\partial D$ such that $\pi_*(W)=0$. 

\medskip

\textbf{Remark}: From the above definition, the notion of  horizontal lift with respect to the Levi-form on the boundary is compatible with the usual notion of horizintal lift if we only consider the category of tangent vector fields on $\partial D$.
Same as before, we shall also define the notion of geodesic curvature in the following sense:

\begin{definition}\label{de:geodesic-c-bdy} Assume that $D$ satisfies $\mathbf{A1}$ and $\mathbf{A2}$ and for each $j$,  $\partial/\partial t^j$ has a horizontal lift to $\partial D$ with respect to the Levi-form $i\partial\dbar\rho$ on $\partial D$. Then
we call
\begin{equation}\label{eq:cjkrho}
    \theta_{j\bar k}(\rho):=\langle V_j^\rho, V_k^\rho\rangle_{i\partial\dbar\rho},
\end{equation}
the geodesic curvature of $\{\partial D_t\}_{t\in\mathbb B}$ with respect to the Levi form $i\partial\dbar\rho$ on $\partial D$.
\end{definition}

Now a natural question is whether each base vector field has a horizontal lift (with respect to the Levi-form) to $\partial D$ or not. We have the following  lemma:

\begin{lemma}[Key Lemma]\label{le:key-lemma} Assume that $D$ satisfies $\mathbf{A1}$ and $\mathbf{A2}$. Put
\begin{equation}
\omega:=i\partial\dbar(-\log-\rho), 
\end{equation} 
where $\rho$ is the defining function in $\mathbf{A2}$. For each $j$, let $V_j$ be the horizontal lift (on $D$) of  $\partial/\partial t^j$ with respect to $\omega$. Then each $V_j$ is smooth up to the boundary of $D$ and $V_j|_{\partial D}$ is horizontal with respect to the Levi form $i\partial\dbar\rho$ on $\partial D$. In particular, every smooth base vector field has a unique smooth horizontal lift with respect to the Levi form and the geodesic curvature $\theta_{j\bar k}(\rho)$ is well defined on $\partial D$.
\end{lemma}

As a generalization of Definition 1.2, we shall introduce the following definition:

\begin{definition}\label{de:inter-psc} Assume that $D$ satisfies $\mathbf{A1}$ and $\mathbf{A2}$. We call $\{D_t\}_{t\in\mathbb B}$ an interpolation family in $\mathcal X$ if $\theta_{j\bar k}(\rho) \equiv 0$ on $\partial D$.
\end{definition}

\textbf{Remaining term in H\"ormander's $L^2$-estimate (fibration with boundary)}: For each $1\leq j\leq m$, let $u_j$ be a smooth section of $\mathcal H$ (see Definition \ref{de:smooth}). 
Put
\begin{equation}\label{eq:cb}
    c=\sum (V_j ~ \lrcorner ~\Theta(\mathcal L,h))|_{D_t} \wedge u_j, \ b=\sum(\dbar V_j)|_{D_t} ~\lrcorner~ u_j,
\end{equation}
where each $V_j$ is the vector field in Lemma \ref{le:key-lemma}. Let $a$ be the $L^2$-minimal solution of
\begin{equation*}
    \dbar^t(\cdot)=\partial^t_\phi b+c,
\end{equation*}
in $L^2(D_t, K_{D_t}+L_t)$. Put
\begin{equation*}
    \omega^t:=i\partial\dbar(-\log-\rho)|_{D_t}.
\end{equation*}Then we shall define
\begin{equation}\label{eq:new-3}
    R:=||c||^2_{i\Theta(\mathcal L,h)|_{D_t}}+||b||^2_{\omega^t}-||a||^2_{\omega^t},
\end{equation}
if $i\Theta(\mathcal L,h)|_{D_t}>0$; and define
\begin{equation*}
    R:=||b||^2_{\omega^t}-||a||^2_{\omega^t},
\end{equation*}
if $\Theta(\mathcal L,h)\equiv 0$. We shall show in Theorem \ref{th:L2} that $R$ is always non-negative.

\subsection{Main theorem} 

\begin{theorem}\label{th:CF} Assume that $D$ satisfies $\mathbf{A1}$ and $\mathbf{A2}$. If
\begin{equation}\label{eq:CF)}
    i\Theta(\mathcal L,h)|_{D_t}>0, \ \forall \ t\in\mathbb B, \ \text{or} \ \Theta(\mathcal L,h)\equiv 0 \ \text{on}\ D,
\end{equation}
then, using the above notation, see \eqref{eq:new-1}, \eqref{eq:cjkh}, \eqref{eq:cjkrho}, \eqref{eq:new-3}, we have the following curvature formula of $\mathcal H$:
\begin{equation}\label{eq:final-CF}
    \sum(\Theta_{j\bar k} u_j, u_k)= \sum\int_{\partial D_t} \theta_{j\bar k}(\rho) \langle u_j,u_k\rangle d \sigma
    + \sum(c_{j\bar k}(h) u_j, u_k) + R,
\end{equation}
where $\langle \cdot, \cdot \rangle$ denotes the point-wise inner product with respect to $i\partial\dbar\rho|_{D_t}$ and $h$, and the surface measure $d\sigma$ with respect to $i\partial\dbar\rho|_{D_t}$ is defined as
\begin{equation*}
    d\sigma:= \frac{\sum \rho_{\bar\lambda}\rho^{\bar\lambda\nu}\partial/\partial \mu^\nu}{\sum \rho_{\bar\lambda}\rho^{\bar\lambda\nu}\rho_{\nu}} ~ \lrcorner ~ \frac{(i\partial\dbar\rho|_{D_t})^n}{n!}.
\end{equation*}
\end{theorem}

\subsection{Applications} \ \

\medskip

We shall show how to use our main theorem to study the complex geometry counterparts of Theorem 1.4. Let us give some positive-curvature criterion of $\mathcal H$ first. Recall that, $\mathcal H$ is said to be semi-positive in the sense of Nakano if $\sum(\Theta_{j\bar k} u_j, u_k) \geq 0$, for all smooth sections $u_1, \cdots, u_m$ of $\mathcal H$. As a direct consequence of our main theorem, we shall prove that:

\begin{corollary}\label{co:nakano} Assume that $D$ satisfies $\mathbf{A1}$ and $\mathbf{A2}$. If $D$ is Stein and $  i\Theta(\mathcal L,h)\geq 0$ on $D$ then $\mathcal H$ is semi-positive in the sense of Nakano.
\end{corollary}

Another very useful notion of positivity is the Griffiths positivity. Recall that $\mathcal H$ is said to be Griffiths semi-positive if $\sum(\Theta_{j\bar k} u, u) \xi_j\bar\xi_k \geq 0$, for every smooth section $u$ of $\mathcal H$ and every $\xi\in\C^m$. It is known that a \emph{finite rank} vector bundle is Griffiths semi-positive if and only if its dual bundle is Griffiths semi-negative. Moreover, the following is true:

\medskip

\textbf{Criterion for Griffiths semi-positivity}: A \emph{finite rank} vector bundle is Griffiths semi-positive if and only if the $\log$-norm of the holomorphic sections of its dual bundle are plurisubharmonic. 

\medskip

Then a natural question is whether there is a similar criterion in our case, i.e. for the infinite rank vector bundle $\mathcal H$. As a first step, we have to define the notion of the \emph{holomorphic section of the \textbf{dual} of $\mathcal H$.} 

\begin{definition}\label{de:smooth-dual-new} For each $t\in \mathbb B$, let $f^t$ be a $\C$-linear mapping from $\mathcal H_t$ to $\C$. We call $f: t \mapsto f^t$ a smooth section of the \textbf{dual} of $\mathcal H$ if there exists a smooth section, say $P(f)$, of $\mathcal H$ such that
\begin{equation}
f^t(u^t)=(u^t, P(f)^t),
\end{equation}
for every $u^t\in \mathcal H_t$ and every $t\in\mathbb B$. We shall write the norm of $f^t$ as $||f^t||:=||P(f)^t||$.
\end{definition}

\begin{definition}\label{de:holomorphic-dual-new} Let $f: t \mapsto f^t$ be a smooth section of the \textbf{dual} of $\mathcal H$. We call $f$ a holomorphic section if 
\begin{equation}
t \mapsto f^t(u^t)
\end{equation}
is holomorphic for every holomorphic section $u$ of $\mathcal H$.
\end{definition}

\textbf{Remark:} Inspired by \cite{BL14}, we shall give a careful study of those holomorphic sections of the dual of $\mathcal H$ defined by a family \emph{currents with compact support} in fibres. 

\medskip

Now we are ready to state the main application of our main theorem:

\begin{corollary}\label{co:dual-psh} Assume that $D$ satisfies $\mathbf{A1}$ and $\mathbf{A2}$. If $D$ is Stein and $i\Theta(\mathcal L,h)\geq 0$ on $D$ then
\begin{equation}
\log||f||:  t \mapsto \log ||f^t||
\end{equation} 
is plurisubharmonic for every holomorphic section $f$ of the dual of $\mathcal H$.
\end{corollary}

\textbf{Remark:} If we choose $f^t$ as a  fixed Dirac measure then we get the plurisubharmonicity of the Bergman kernel \cite{Bern06}. In section 5, we shall also use Corollary \ref{co:dual-psh} to study variation of the deriatives of the Bergman kernel.

\medskip

\textbf{Triviality and flatness:} In case every fibre $D_t$ is compact without boundary, we know that the criterions for flatness of $\mathcal H$ are quite usefull in the study of the uniqueness problems of extremal K\"ahler metrics (see \cite{Bern09a} and \cite{Bern14}). In our case, we call $\mathcal H$ a \emph{flat bundle} if 
\begin{equation}
\Theta_{j\bar k} u \equiv 0, \ \forall \ 1\leq j,k\leq m.
\end{equation}
for every smooth section $u$ of $\mathcal H$. From Theorem 1.4, one may guess that flatness of $\mathcal H$ should be related to triviality of the fibration $\pi$. Let us introduce the following definition as a generalization of Definition 1.3.

\begin{definition}\label{de:trivial} Assume that $D$ satisfies $\mathbf{A1}$ and $\mathbf{A2}$. We call $\{D_t\}_{t\in\mathbb B}$ a trivial family, or $D$ is trivial, if there exists a biholomorphic mapping $\Phi: D_0\times \mathbb B \to D$ such that 
\begin{equation}
\Phi(D_0\times\{t\})=D_t, \ \forall \ t\in\mathbb B,
\end{equation}
and $\Phi_{*}(\partial/\partial t^j)$ extends to a smooth $(1,0)$-vector field on $\mathcal X$ for every $1\leq j\leq m$.
\end{definition}

The following theorem can be seen as a generalization of Proposition 1.5:

\begin{theorem}\label{th:flat} Assume that $D$ satisfies $\mathbf{A1}$ and $\mathbf{A2}$. Assume further that the total space $D$ is Stein. If $K_{\mathcal X/\mathbb B}+\mathcal L$ is trivial on each fibre of $\pi$ and $\Theta(\mathcal L,h)\equiv 0$ on $D$ then flatness of $\mathcal H$ and triviality of $D$ are equivalent.
\end{theorem}

As a direct corollary of Theorem \ref{th:flat}, we have

\begin{corollary}\label{co:last} Let $D_0$ be a smooth domain in $\C$. Let 
\begin{equation}
F:(t,z)\mapsto(t,z+a(t)\bar z)
\end{equation}
be a mapping from $\mathbb B\times D_0$ to $\mathbb B\times \C$. Assume that $a$ is holomorphic on $\mathbb B$ and
\begin{equation}
|a| <1, \ \text{on}\ \mathbb B, \  a(0)=0.
\end{equation}
Then $\{F(\{t\}\times D_0)\}_{t\in\mathbb B}$ is a trivial family if and only if $a\equiv0$ on $\mathbb B$.
\end{corollary}

\textbf{Remark:} One may also give a direct proof of Corollary \ref{co:last} by introducing the notion of Kodaira-Spencer class for deformations with boundary, we leave it to the interested reader.

\section{Geodesic curvature and interpolation family}

In this section, we shall give two proofs of Lemma \ref{le:key-lemma} and show that our Definition of interpolation family (see Definition \ref{de:inter-psc}) is  compatible with the usual definition of interpolation family of Hermitian norms on $\mathbb C^n$ for $n\geq 1$.

\medskip

\textbf{Relation with Levi-flatness}:  By the definition of the geodesic curvature $\theta_{j\bar k}(\rho)$ (see Definition \ref{de:geodesic-c-bdy}) of $\{\partial D_t\}$, we have:   

\medskip

\emph{ Assume that every fibre $D_t$ is one dimensional. Then $\{D_t\}$ is an interpolation family if and only if the boundary of $D$ is Levi flat.}  

\medskip

For higher fibre dimension case, the criterion for interpolation family (see \cite{Semmes88} and references therein) is not so obvious. We will give a study of it in section 3.3. Let us prove our Key Lemma first.

\subsection{First proof of Lemma \ref{le:key-lemma}}\ \

\medskip

Since $V_j$ is a lift of $\partial/\partial t^j$, locally one may write,
\begin{equation}\label{eq:vj-1}
    V_j=\partial/\partial t^j-\sum v_j^\lambda\partial/\partial \mu^\lambda.
\end{equation}
Put $\psi=-\log-\rho$. By definition, we know that each $V_j$ is determined by
\begin{equation*}
    \langle V_j, \partial/\partial \mu^\nu\rangle_{i\partial\dbar\psi}\equiv 0, \ \text{on} \ D, \ \forall \ 1\leq \nu\leq n,
\end{equation*}
thus
\begin{equation}\label{eq:vj-2}
    v_j^\lambda=\sum\psi_{j\bar\nu}\psi^{\bar\nu\lambda}, \ \text{on} \ D.
\end{equation}

\textbf{Fibre dimension one case}: If $n=1$, by direct computation, we have
\begin{equation}\label{eq:keylemma}
    V_j:=\frac{\partial}{\partial t^j}-\frac{\rho_j\rho_{\bar\mu}-\rho\rho_{j\bar\mu}}{|\rho_\mu|^2-\rho\rho_{\mu\bar\mu}}
    \frac{\partial}{\partial \mu}.
\end{equation}
By our assumption $\mathbf{A2}$, $\rho_\mu$ has no zero point near the boundary and $\rho_{\mu\bar\mu} >0$ near the boundary, thus $V_j$ is smooth up to the boundary of $D$. Furthermore, \eqref{eq:keylemma} implies that $V_j(\rho)=0$ on $\{\rho=0\}$. Notice that, in case $n=1$, every tangent vector field on $\partial D$ is horizontal with respect to the Levi-form on $\partial D$. Thus we know that if $n=1$ then  $V_j|_{\partial D}$ is the horizontal lift of $\partial/\partial t^j$ with respect to the Levi-form.  

\medskip

\textbf{General case}: The general case can also be proved by direct computation (see the second proof below). But there is also a simple proof as follows: If $n\geq2$, fix $x_0\in\partial D_0$, then by our assumption $\mathbf{A2}$, $(\rho_{\lambda\bar \nu})(x_0)$ is a positive definite matrix, thus one may choose local coordinates around $x_0$ such that
\begin{equation*}
    \left(\rho_{\lambda\bar \nu}(x_0)\right)=I_n, \ \rho_\nu(x_0)=0, \ \forall \ \nu\geq2,
\end{equation*}
where $I_n$ is the identity matrix. Now we have
\begin{equation*}
    v_j^1(x_0)=\frac{\rho_j\rho_{\bar1}-\rho\rho_{j\bar1}}{|\rho_1|^2-\rho}(x_0)=\frac{\rho_j}{\rho_1}(x_0), \  v_j^\lambda(x_0)=\rho_{j\bar \lambda}(x_0), \ \forall \ \lambda\geq2.
\end{equation*}
By assumption $\mathbf{A2}$, we know that $\rho_1(x_0)\neq 0$, thus $V_j$ is smooth up to the boundary and
\begin{equation*}
    V_j(\rho)(x_0)=\rho_j(x_0)-\sum v_j^\lambda\rho_\lambda(x_0)=\rho_j(x_0)-\rho_j(x_0)=0.
\end{equation*}
Moreover, we have
\begin{equation*}
    \langle V_j, \partial/\partial \mu^\lambda \rangle_{i\partial\dbar\rho}(x_0)=\rho_{j\bar\lambda}(x_0)-v_j^{\lambda}(x_0)=0 , \ \forall \ \lambda\geq 2,
\end{equation*}
which implies that each $V_j$ is horizontal with respect to the Levi form. The proof is complete.

\subsection{Second proof of Lemma \ref{le:key-lemma}}\ \

\medskip

In this subsection, we will give an explicit formula for each $V_j$. Here we shall use some computations from \cite{Choi15}. Put 
\begin{equation*}
    \rho^\alpha=\sum \rho^{\bar\beta\alpha}\rho_{\bar\beta}, \ \ \ |\partial\rho|^2=\sum|\rho_{\alpha}|^2.
\end{equation*}
By \eqref{eq:vj-1} and \eqref{eq:vj-2}, we have
\begin{equation}\label{eq:vj-4}
    V_j=\partial/\partial t^j-\sum\psi_{j\bar\nu}\psi^{\bar\nu\lambda}  \partial/\partial \mu^\lambda,
\end{equation}
where $\psi=-\log-\rho$. One may verify that (see \cite{Choi15})
\begin{equation}\label{eq:vj-5}
    \psi^{\bar\alpha\beta}=(-\rho)\left(   \rho^{\bar\alpha\beta}+\frac{\rho^{\bar\alpha}\rho^{\beta}}{\rho-|\partial\rho|^2}
     \right).
\end{equation}
By direct computation, we have
\begin{equation}\label{eq:vj-6}
    \sum\psi_{j\bar\alpha}\psi^{\bar\alpha\beta}=-\frac{\rho^{\beta}\rho_{j}}{\rho-|\partial\rho|^2}+
    \sum \left(\rho^{\bar\alpha\beta}\rho_{j\bar\alpha}+
    \frac{\rho^{\bar\alpha}\rho^{\beta}\rho_{j\bar\alpha}}{\rho-|\partial\rho|^2}\right).
\end{equation}
From \eqref{eq:vj-6}, we know that each $V_j$ is smooth up to the boundary of $D$ and is tangent to the boundary of $D$. By a direct computation, we also have that each $V_j$ is horizontal with respect to the Levi-form of the boundary of $D$. Thus the proof of Lemma \ref{le:key-lemma} is complete.

\subsection{Geodesic curvature for $\{\partial D_t\}$}\ \

\medskip

Denote by $\hat V_j$ the horizontal lift of $\partial/\partial t^j$ with respect to $i\partial\dbar\rho$. By the proof of \eqref{eq:vj-2}, we have
\begin{equation}\label{eq:vj-7}
    \hat V_j= \frac{\partial}{\partial t^j}- \sum \rho_{j\bar\alpha}  \rho^{\bar\alpha\beta}\frac{\partial}{\partial \mu^{\beta}} .
\end{equation}
By \eqref{eq:vj-4} and \eqref{eq:vj-6} and a direct computation, we get
\begin{equation}\label{eq:vj-8}
    \theta_{j\bar k} (\rho)=\langle V_j, V_k\rangle_{i\partial\dbar\rho} =c_{j\bar k}(\rho)+\frac{|\partial \rho|^2 \hat V_j(\rho)\overline{\hat V_k(\rho)} }{(\rho-|\partial\rho|^2)^2},
\end{equation}
where
\begin{equation}\label{eq:vj-9}
    c_{j\bar k}(\rho):= \langle \hat V_j, \hat V_k\rangle_{i\partial\dbar\rho}.
\end{equation}
Thus we have:

\begin{proposition}\label{pr:theta-c} Let $\{D_t\}$ be a smooth family of smoothly bounded Stein domains. Then \begin{equation}
\sum \theta_{j\bar k}(\rho) \xi^j\bar\xi^k \geq \sum c_{j\bar k}(\rho)\xi^j\bar\xi^k, \ \ \forall \ \xi\in \C^m,
\end{equation}
on $\partial D$. Moreover, $\theta_{j\bar k}(\rho)\equiv c_{j\bar k}(\rho)$ if and only if each $\hat V_j$ is tangent to the boundary of $D$.
\end{proposition}

\subsection{Relation with interpolation of norms}\ \

\medskip

Let $h$ be a smooth Hermitian norm on the trivial vector bundle $\mathbb B\times \mathbb C^n$. Then for each $t\in\mathbb B$, $h^t:=h|_{t\times\mathbb C^n}$ defines a Hermitian norm on $\mathbb C^n$. It is known that $\{h^t\}$ defines an interpolation family if and only if the curvature of $h$ vanishes identically on $\mathbb B$ (see Semmes \cite{Semmes88}). Denote by $N_t$ the unit ball in $\mathbb C^n$ defined by $h^t$. We shall prove that:

\begin{proposition}\label{pr:interpolation} $\{h^t\}$ defines an interpolation family if and only if the geodesic curvature of $\{\partial N_t\}$ vanishes identically.
\end{proposition}

\begin{proof} Let us write 
\begin{equation*}
   h^t(z)=\sum h_{\alpha\bar\beta}(t)z^\alpha\bar z^{\beta}.
\end{equation*}
By definition, 
\begin{equation*}
   \rho(t,z):=h^t(z)-1
\end{equation*}
is a defining function for $\{\partial N_t\}$. By direct computation, we have that $\hat V_j(\rho)$ vanishes identically. Thus by Proposition \ref{pr:theta-c}, the geodesic curvature, $\theta_{j\bar k}(\rho)$, of $\{\partial N_t\}$ is equal to $c_{j\bar k}(\rho)$. By \eqref{eq:vj-9}, 
\begin{equation}\label{eq:vj-10}
    c_{j\bar k}(\rho)= \langle \hat V_j, \hat V_k\rangle_{i\partial\dbar\rho}=\rho_{j\bar k}-\sum \rho_{j\bar\alpha}  \rho^{\bar\alpha\beta}\rho_{\bar k\beta},
\end{equation}
thus we have
\begin{equation}\label{eq:vj-11}
    \theta_{j\bar k}(\rho)= \sum \left(h_{\alpha\bar\beta,j\bar k}-\sum h_{\alpha\bar\lambda,j}h^{\bar\lambda\nu}h_{\nu\bar\beta,\bar k}      \right) z^{\alpha}\bar z^{\beta}.
\end{equation}
Thus $\theta_{j\bar k}(\rho)$ vanishes identically if and only if 
\begin{equation}\label{eq:vj-12}
h_{\alpha\bar\beta,j\bar k}-\sum h_{\alpha\bar\lambda,j}h^{\bar\lambda\nu}h_{\nu\bar\beta,\bar k}  \equiv 0, \ \ \text{on}\ \mathbb B.
\end{equation}
Notice that \eqref{eq:vj-12} is equivalent to that the curvature of $h$ vanishes identically. The proof is complete.
\end{proof}

\section{Curvature formula}

\subsection{Definition of the Chern connection}\ \

\medskip

By Definition \ref{de:cc}, it suffices to find a linear operator $D_{t^j}$ from $\Gamma (\mathcal H)$ to $\Gamma (\mathcal H)$ such that
\begin{equation}\label{eq:chern-connection-1}
    \partial_{t^j}(u,v)=(D_{t^j} u, v) + (u, \dbar_{t^j} v), \ \forall\ u, v\in \Gamma(\mathcal H),
\end{equation}
is true. By Definition \ref{de:re}, the left hand side of \eqref{eq:chern-connection-1} can be written as
\begin{equation}\label{eq:variation}
     \partial_{t^j} (\pi_*(c_n\{\textbf{u},\textbf{v}\})),
\end{equation}
where $\mathbf u, \mathbf v$ are arbitrary representatives (see Definition \ref{de:re}) of $u,v$ and $c_n= i^{n^2}$ such that $c_n\{\textbf{u},\textbf{u}\}$ is a positive $(n,n)$-form on the total space. 

Assume that $D$ satisfies $\mathbf{A1}$ and $\mathbf{A2}$. Let $V_j$ be the vector fields in Lemma \ref{le:key-lemma}. Since $V_j(\rho)=0$ on $\partial D$, by Corollary \ref{co:vfi}, we have
\begin{equation}\label{eq:Lie-first}
    \partial_{t^j} (\pi_*\{\textbf{u},\textbf{v}\})=\pi_*(L_{V_j}\{\textbf u, \textbf v\}).
\end{equation}
Let $d^{\mathcal L}$ be the Chern connection on $\mathcal L$. Then we have
\begin{equation*}
    d\{\textbf u, \textbf v\}=\{d^{\mathcal L} \textbf u, \textbf v\}+(-1)^n \{\textbf u,d^{\mathcal L} \textbf v\}.
\end{equation*}
Using Cartan's formula,
\begin{equation*}
L_{V_j}=d\delta_{V_j}+\delta_{V_j}d,
\end{equation*}
we get that
\begin{equation*}
    L_{V_j}\{\textbf u, \textbf v\}=\{L_j \textbf u, \textbf v\}+\{\textbf u,L_{\bar j}\textbf v\},
\end{equation*}
where
\begin{equation*}
    L_j :=d^{\mathcal L} \delta_{V_j}+\delta_{V_j}d^{\mathcal L},
\end{equation*}
and
\begin{equation*}
    L_{\bar j}:=d^{\mathcal L} \delta_{\bar V_j}+\delta_{\bar V_j}d^{\mathcal L}.
\end{equation*}
Since $ \textbf v$ is an $(n,0)$-form, we have
\begin{equation}\label{eq:chern01}
    L_{\bar j} \textbf v=\delta_{\bar V_j} \dbar \textbf v,
\end{equation}
By \eqref{eq:def-dbar25}, we know that $ L_{\bar j} \textbf v$ is a representative of  $\dbar_{t^j} v$. Thus we have
\begin{equation}\label{eq:chern-connection-11}
    \partial_{t^j}(u,v)=\pi_*(c_n\{L_j \textbf{u},\textbf{v}\}) + (u, \dbar_{t^j} v), \ \forall\ u, v\in \Gamma(\mathcal H).
\end{equation}
Notice that the $(n,0)$-part of $L_j \textbf  u$ can be written as
\begin{equation}\label{eq:(n,0)}
(\partial_\phi \delta_{V_j} +\delta_{V_j}\partial_{\phi}) \textbf  u,
\end{equation}
where $\partial_\phi$ denotes the $(1,0)$-component of $d^{\mathcal L}$. Thus we have
\begin{equation*}
\pi_*(c_n\{L_j \textbf{u},\textbf{v}\})(t) =\left(i_t^*(\partial_\phi \delta_{V_j} +\delta_{V_j}\partial_{\phi}) \textbf  u, v^t\right).
\end{equation*}
Our assumption $\mathbf{A2}$ implies that
\begin{equation}
\{v^t: v\in\Gamma(\mathcal H)\}.
\end{equation} 
is dense in the Hilbert space $\mathcal H_t$ (see the proof of Lemma \ref{le:extension} below). Thus there is a \textbf{unique} element, say $\sigma^t$, in $\mathcal H_t$ such that
\begin{equation}\label{eq:riesz-repre}
\left(i_t^*(\partial_\phi \delta_{V_j} +\delta_{V_j}\partial_{\phi}) \textbf  u, v^t\right)=(\sigma^t, v^t),
\end{equation}
which implies that there is a \textbf{unique} element, $\sigma^t$, in $\mathcal H_t$ such that
\begin{equation}\label{eq:chern-connection-111234}
    \partial_{t^j}(u,v)=(\sigma^t, v^t) + (u, \dbar_{t^j} v), \ \forall\ u, v\in \Gamma(\mathcal H).
\end{equation}
Thus by Definition \ref{de:cc}, we know that: $\mathcal H$ has a Chern connection if and only if 
\begin{equation*}
\sigma: t\to \sigma^t,
\end{equation*}
defines a smooth section of $\mathcal H$, i.e. 
\begin{equation*}
\sigma \in \Gamma(\mathcal H).
\end{equation*}
By \eqref{eq:riesz-repre}, $\sigma^t$ is the Bergman projection to $\mathcal H_t$ of $i_t^*((\partial_\phi \delta_{V_j} +\delta_{V_j}\partial_{\phi}) \textbf  u)$. Thus by Hamilton's theorem (see \cite{Hamilton77}, \cite{Hamilton79}, \cite{GreeneK82} or Appendix \ref{ss:SBK}), if $\{D_t\}$ is a smooth family of smoothly bounded Stein domains then $\sigma\in \Gamma(\mathcal H)$ and
\begin{equation}\label{eq:41000}
D_{t^j}u=\sigma.
\end{equation}
Thus we have:

\begin{proposition}\label{pr:chern} If $D$ satisfies $\mathbf{A1}$ and $\mathbf{A2}$ then the Chern connection is well defined on $\mathcal H$.
\end{proposition}

Now we are ready to compute the curvature of the Chern connection of $\mathcal H$. First we shall show how to get a curvature formula for \textbf{holomorphic} sections of $\mathcal H$. 

\subsection{Curvature formula for holomorphic sections}\ \

\medskip

 Let $u_1, \cdots, u_m$ be holomorphic sections (see Definition \ref{de: holo-section-H}) of $\mathcal H$. By definition of the Chern connection and \eqref{eq:new-1}, we have
\begin{equation}\label{eq:curvature2}
   (\Theta_{j\bar k} u_j,u_k)=(D_{t^j}u_j,D_{t^k}u_k)- (u_j,u_k)_{j\bar k}. 
\end{equation}
By \eqref{eq:chern-connection-11}, we have
\begin{equation}\label{eq:jbark-11}
    (u_j,u_k)_{j\bar k}=\dbar_{t^k}\pi_*(c_n\{L_j \textbf u_j,\textbf u_k\})= \pi_*(c_n\{L_j \textbf u_j,L_k\textbf u_k\})+ \pi_*(c_n\{L_{\bar k} L_j \textbf u_j,\textbf u_k\}).
\end{equation}
Since each $u_j$ is a holomorphic section, we have $i_t^*(L_{\bar k} \textbf u_j)=(\dbar_{t^k}u_j)(t)\equiv 0$. Thus 
\begin{equation}\label{eq:jbark-1}
   \pi_*\{L_j L_{\bar k}  \textbf u_j,\textbf u_k\}=
    \partial_{t^j}   \pi_*\{L_{\bar k}\textbf u_j,\textbf u_k\}
    -\pi_*\{L_{\bar k}\textbf u_j,L_{\bar j}\textbf u_k\}\equiv 0,
\end{equation}
which implies that
\begin{equation}\label{eq:jbark}
    (u_j,u_k)_{j\bar k}=\pi_*(c_n\{L_j \textbf u_j,L_k\textbf u_k\}) -\pi_*(c_n\{[ L_j, L_{\bar k}]\textbf u_j, \textbf u_k\}).
\end{equation}
We shall use the following formula:

\begin{proposition}\label{pr:key} $ [L_j,L_{\bar k}]= d^{\mathcal L} \delta_{[V_j,\bar V_k]}+\delta_{[V_j,\bar V_k]} d^{\mathcal L}+ \langle V_j, V_k\rangle_{i\Theta(\mathcal L,h)}.$
\end{proposition}

\begin{proof} By definition, locally we have
\begin{equation*}
    L_j=L_{V_j}-V_j(\phi),  \ L_{\bar k}= L_{\bar V_k}.
\end{equation*}
Thus
\begin{equation*}
    [L_j,L_{\bar k}]=[L_{V_j}, L_{\bar V_k}] -[V_j(\phi),L_{\bar V_k}]=L_{[V_j,\bar V_k]}+ \bar V_k V_j(\phi).
\end{equation*}
By Cartan's formula, we have
\begin{equation*}
    L_{[V_j,\bar V_k]}=d^{\mathcal L} \delta_{[V_j,\bar V_k]}+\delta_{[V_j,\bar V_k]} d^{\mathcal L}+ \delta_{[V_j,\bar V_k]} \partial\phi,
\end{equation*}
Thus
\begin{equation*}
    [L_j,L_{\bar k}]-\left(d^{\mathcal L} \delta_{[V_j,\bar V_k]}+\delta_{[V_j,\bar V_k]} d^{\mathcal L}\right)=
    \delta_{[V_j,\bar V_k]} \partial\phi+\bar V_k V_j(\phi).
\end{equation*}
By direct computation, we have
\begin{equation}\label{eq:strange}
    \delta_{[V_j,\bar V_k]} \partial\phi+\bar V_k V_j(\phi)=\langle V_j, V_k\rangle_{i\Theta(\mathcal L,h)}.
\end{equation} 
Thus this proposition follows.
\end{proof}

Now we can prove the following:

\begin{lemma}\label{le:boundary} If $i\Theta(\mathcal L,h)|_{D_t}>0$, $\forall \ t\in \mathbb B$, then
\begin{equation}\label{eq:boundary^1}
    \sum \pi_*(c_n\{[ L_j, L_{\bar k}]\textbf u_j, \textbf u_k\})=||c||^2_{i\Theta(\mathcal L,h)|_{D_t}}+ B+ \sum(c_{j\bar k}(h) u_j, u_k),
\end{equation}
where $c$ is defined by \eqref{eq:cb} and $B$ is the boundary term defined by
\begin{equation*}
B:=\sum\int_{\partial D_t} \theta_{j\bar k}(\rho) \langle u_j,u_k\rangle d \sigma 
\end{equation*} If $i\Theta(\mathcal L,h) \equiv 0$ on $D$ then
\begin{equation}\label{eq:boundary^2}
     \sum \pi_*(c_n\{[ L_j, L_{\bar k}]\textbf u_j, \textbf u_k\})=B.
\end{equation}
\end{lemma}

\begin{proof} By the above proposition, we have
\begin{equation}\label{eq:Lie-bern-1}
     \sum \pi_*(c_n\{[ L_j, L_{\bar k}]\textbf u_j, \textbf u_k\})=\sum c_n\int_{\partial D_t}\{\delta_{[V_j,\bar V_k]}u_j, u_k\}+I,
\end{equation}
where
\begin{equation}
I:=\sum (\langle V_j, V_k\rangle_{i\Theta(\mathcal L,h)} u_j, u_k).
\end{equation}
Now the boundary term can be written as
\begin{equation*}
    \sum c_n\int_{\partial D_t}\{\delta_{[V_j,\bar V_k]}u_j, u_k\}= \sum \int_{\partial D_t} ( \delta_{[V_j,\bar V_k]}\partial \rho) \langle u_j, u_k \rangle d\sigma.
\end{equation*}
We shall prove $\delta_{[V_j,\bar V_k]} \partial\rho = \theta_{j\bar k}(\rho)$ on $\partial D$. In fact, by \eqref{eq:strange}, we have
\begin{equation*}
\delta_{[V_j,\bar V_k]} \partial\rho+\bar V_k V_j(\rho)=\langle V_j, V_k\rangle_{i\partial\dbar\rho},
\end{equation*}
and by our key lemma, $V_j|_{\partial D}=V_j^\rho$, thus $\bar V_k V_j(\rho)\equiv 0$ on $\partial D$ and
\begin{equation*}
    \delta_{[V_j,\bar V_k]} \partial\rho \equiv \theta_{j\bar k}(\rho),\ \text{on}\ \partial D.
\end{equation*}
Thus
\begin{equation}\label{eq:4.19}
B=\sum c_n\int_{\partial D_t}\{\delta_{[V_j,\bar V_k]}u_j, u_k\}.
\end{equation}
and \eqref{eq:Lie-bern-1} implies \eqref{eq:boundary^2}.
Now let us prove \eqref{eq:boundary^1}: By \eqref{eq:cjkh}, we have
\begin{equation*}
\langle V_j, V_k\rangle_{i\Theta(\mathcal L,h)}=c_{j\bar k}(h)+\langle V_j-V_j^h, V_k-V_k^h\rangle_{i\Theta(\mathcal L,h)|_{D_t}}.
\end{equation*}
Since 
\begin{equation*}
    (V_j-V_j^h)~ \lrcorner ~ ( i\Theta(\mathcal L,h)|_{D_t})= ((V_j-V_j^h)~ \lrcorner ~i\Theta(\mathcal L,h))|_{D_t}=(V_j~ \lrcorner ~i\Theta(\mathcal L,h))|_{D_t},
\end{equation*}
by \eqref{eq:cb}, we have
\begin{equation}\label{eq:Lie-bern-2}
I =\sum (c_{j\bar k}(h)u_j, u_k)+ ||c||^2_{i\Theta(\mathcal L,h)|_{D_t}}.
\end{equation}
Thus \eqref{eq:boundary^1} follows.
\end{proof}

By \eqref{eq:curvature2} and \eqref{eq:jbark}, we have
\begin{equation}\label{eq:curvature21}
   (\Theta_{j\bar k} u_j,u_k)=\pi_*(c_n\{[ L_j, L_{\bar k}]\textbf u_j, \textbf u_k\}) +(D_{t^j}u_j,D_{t^k}u_k)-\pi_*(c_n\{L_j \textbf u_j,L_k\textbf u_k\}) .
\end{equation}
Let $a^j$ be the  $(n,0)$-part of $i_t^*(L_j \textbf u_j)$ and  $b^j$ be the $(n-1,1)$-part of $i_t^*(L_j \textbf u_j)$, i.e.
\begin{equation*}
    a^j=i_t^*(\partial_\phi \delta_{V_j}+\delta_{V_j}\partial_\phi) \textbf u_j=i_t^*[\partial_\phi,\delta_{V_j}] \textbf u_j,
\end{equation*}
and
\begin{equation*}
    b^j=i_t^*(\dbar \delta_{V_j}+\delta_{V_j}\dbar)\textbf u_j=(\dbar V_j)|_{D_t} ~ \lrcorner ~ u_j.
\end{equation*}
Then we have
\begin{equation}\label{eq:curvature22}
||\sum D_{t^j}u_j||^2-\sum \pi_*(c_n\{L_j \textbf u_j,L_k\textbf u_k\})=-||a||^2-\pi_*(c_n\{b,b\}),
\end{equation}
where $b$ is defined in \eqref{eq:cb} and
\begin{equation*}
    a=\sum ((D_{t^j}u_j)(t)-a^j).
\end{equation*}
We shall prove that:

\begin{proposition}\label{pr:abc-fun}  $a$ is the $L^2$-minimal solution of
\begin{equation}\label{eq:dbar-equation-3}
  \dbar^t(a)=\partial_\phi^t b+c,
\end{equation}
where $b$ and $c$ are defined in \eqref{eq:cb}.
\end{proposition} 

\begin{proof} Since $i_t^*(\dbar \textbf u_j )\equiv 0$ and $i_t^*(\partial_\phi \textbf u_j )\equiv 0$, we have 
\begin{equation}\label{eq:dbar-equation}
    \dbar^t a^j+\partial_\phi^t b^j= i_t^*(\dbar[\partial_\phi,\delta_{V_j}]+\partial_\phi[\dbar, \delta_{V_j}]) \textbf u_j=i_t^*([\dbar,[\partial_\phi,\delta_{V_j}]]+[\partial_\phi,[\dbar, \delta_{V_j}]])\textbf u_j.
\end{equation}
Since
\begin{equation}\label{eq:dbar-equation-1}
    [\dbar,[\partial_\phi,\delta_{V_j}]]+[\partial_\phi,[\dbar, \delta_{V_j}]]+[\delta_{V_j}, [\dbar,\partial_\phi]]\equiv 0,
\end{equation}
and $[\dbar,\partial_\phi]\equiv\Theta(\mathcal L, h)$, we get that
\begin{equation}\label{eq:dbar-equation-2}
    \dbar^t a^j+\partial_\phi^t b^j=-(V_j~\lrcorner~ \Theta(\mathcal L, h))|_{D_t} \wedge u_j.
\end{equation}
Recall that by \eqref{eq:41000}, each $D_{t^j} u_j$ is just the Bergman projection to $\mathcal H_t$ of $a^j$. Thus this proposition follows from \eqref{eq:dbar-equation-2}.
\end{proof}

By Lemma \ref{le:boundary}, \eqref{eq:curvature21} and \eqref{eq:curvature22}, we have
\begin{equation}\label{eq:curvature31}
    \sum(\Theta_{j\bar k} u_j, u_k)= \sum\int_{\partial D_t} \theta_{j\bar k}(\rho) \langle u_j,u_k\rangle d \sigma
    + \sum(c_{j\bar k}(h) u_j, u_k) + R',
\end{equation}
where
\begin{equation*}
    R'=||c||^2_{i\Theta(\mathcal L,h)|_{D_t}}-\pi_*(c_n\{b,b\})-||a||^2.
\end{equation*}
Now let us prove that $R'=R$. It is enough to prove that
\begin{equation*}
    -\pi_*(c_n\{b,b\})=||b||^2_{\omega^t}.
\end{equation*}
But it follows directly from the following lemma:

\begin{lemma}\label{le:primitive} $b$ is primitive with respect to $\omega^t:=i\partial\dbar(-\log-\rho)|_{D_t}$.
\end{lemma}

\begin{proof} Recall that
\begin{equation*}
    b=\sum (\dbar V_j)|_{D_t} ~ \lrcorner ~ u_j.
\end{equation*}
Since $b$ is an $(n-1,1)$-form, by definition of primitivity, it suffices to show that
\begin{equation*}
    \omega^t\wedge b\equiv 0, \ \text{on} \ D_t.
\end{equation*}
Thus it is enough to prove that
\begin{equation}\label{eq:primitive}
    ((\dbar V_j)  ~ \lrcorner ~ i\partial\dbar(-\log-\rho) )|_{D_t}\equiv 0, \ \forall \ 1\leq j\leq m.
\end{equation}
By definition of $V_j$ in our Key-Lemma, $(V_j~ \lrcorner ~ i\partial\dbar(-\log-\rho))|_{D_t}=0$. Thus \eqref{eq:primitive} is true.
\end{proof}

\textbf{Remark}: Now we know that Theorem \ref{th:CF} is true if each $u_j$ is a \textbf{holomorphic section} of $\mathcal H$. For finite rank vector bundles, the curvature operators are always pointwise defined, thus it is enough to find a curvature formula for holomorphic sections in finite rank case. One may guess that the same argument also works for the general infinite rank vector bundle. In the next subsection, we shall prove that at least the curvature operators for our bundle $\mathcal H$ are pointwise defined. Thus we know that \eqref{eq:curvature31} is also true for \textbf{general smooth sections} of $\mathcal H$.

\subsection{Curvature formula for general sections}\ \

\medskip

By the above remark, we need to prove that the curvature operators $\Theta_{j\bar k}$ on $\mathcal H$ are pointwise defined. We shall use the following two lemmas.

\begin{lemma}\label{le:symmetry} $(\Theta_{j\bar k}u,v)=(u, \Theta_{k\bar j} v)$, $\forall \ u,v\in\Gamma(\mathcal H)$.
\end{lemma}

\begin{proof} By \eqref{eq:chern-connection}, we have
\begin{equation*}
    (u,v)_{\bar k j}=\dbar_{t^k}((D_{t^j}u,v)+(u, \dbar_{t^j}v))=(\dbar_{t^k}D_{t^j}u,v)+(D_{t^j}u,D_{t^k}v)+
    (\dbar_{t^k}u, \dbar_{t^j}v)+ (u, D_{t^k}\dbar_{t^j}v).
\end{equation*}
On the other hand,
\begin{equation*}
    (u,v)_{j\bar k}=\partial_{t^j}((\dbar_{t^k}u,v)+(u, D_{t^k}v))=(D_{t^j}\dbar_{t^k}u,v)+(D_{t^j}u,D_{t^k}v)+
    (\dbar_{t^k}u, \dbar_{t^j}v)+ (u, \dbar_{t^j}D_{t^k}v).
\end{equation*}
Since $(u,v)_{\bar k j}\equiv (u,v)_{j\bar k}$, the lemma follows by comparing the difference of the above two equality.
\end{proof}

\begin{lemma}\label{le:extension} Assume that $D$ satisfies $\mathbf{A1}$ and $\mathbf{A2}$. Fix $u\in\Gamma(\mathcal H)$ and $t_0\in\mathbb B$. Then $u|_{D_{t_0}}$ can be approximated by holomorphic sections of $\mathcal H$ in the following sense:

For every $0<s<1$,  there exists a holomorphic section $u^{(s)}$ of $\mathcal H$ over an open neighborhood (may depend on $s$) of $t_0$ such that
\begin{equation}
(\eta,s)\mapsto u^{(s)}|_{D_{t_0}} (\eta),  \ 0<s<1,  \ \ \ (\eta,0)\mapsto u|_{D_{t_0}}(\eta),
\end{equation} 
 is smooth up to the boundary of $D_{t_0}\times[0,1)$.
\end{lemma}

\begin{proof} Fix a sufficiently small $\varepsilon >0$ and consider
\begin{equation*}
    D_{t_0}^s:=\{\zeta\in \pi^{-1}(t_0): \rho(t_0, \zeta)<\varepsilon s\}.
\end{equation*}
Let us define $u^{(s)}$ as the Bergman projection  to the space of $L^2$-holomorphic forms on $D_{t_0}^s$ of  $u|_{D_{t_0}^s}$. By Siu's theorem \cite{Siu76}, for every $0<s<1$, $D_{t_0}^s$ has a Stein neighborhood in $\mathcal X$. Thus by Cartan's theorem, every $u^{(s)}$ extends to a holomorphic section (also denoted by $u^{(s)}$) of $\mathcal H$ over an open neighborhood of $t_0$. The regularity properties of $\{u^{(s)}\}$ follows directly from Hamilton's theorem (see Appendix \ref{ss:SBK}).
\end{proof}

Now let us finish the proof of Theorem \ref{th:CF}. By the above two lemmas, for every $u_j\in\Gamma(\mathcal H)$, $1\leq j\leq m$, $t_0\in\mathbb B$, we have
\begin{eqnarray*}
  (\Theta_{j\bar k}u_j,u_k)(t_0) &=& \lim_{s_1\to 0}(\Theta_{j\bar k}u_j,u_k^{(s_1)})(t_0)=\lim_{s_1\to 0}(u_j,\Theta_{k\bar j}u_k^{(s_1)})(t_0) \\
   &=& \lim_{s_1\to 0}\lim_{s_2\to 0} (u_j^{(s_2)}, \Theta_{k\bar j}u_k^{(s_1)})(t_0)=
    \lim_{s_1\to 0}\lim_{s_2\to 0} (\Theta_{j\bar k} u_j^{(s_2)}, u_k^{(s_1)})(t_0) \\
   &=& \lim_{s\to 0}(\Theta_{j\bar k} u_j^{(s)}, u_k^{(s)})(t_0).
\end{eqnarray*}
where $u_j^{(s)}$ are holomorphic sections of $\mathcal H$ defined in Lemma \ref{le:extension}. By our curvature formula for holomorphic sections, we have
\begin{equation*}
    \sum(\Theta_{j\bar k} u_j^{(s)}, u_k^{(s)})= \sum\int_{\partial D_t} \theta_{j\bar k}(\rho) \langle u_j^{(s)},u_k^{(s)}\rangle d \sigma
    + \sum(c_{j\bar k}(h) u_j^{(s)}, u_k^{(s)}) + R(s),
\end{equation*}
where
\begin{equation*}
    R(s)=||c(s)||^2_{i\Theta(\mathcal L,h)|_{D_t}}+||b(s)||^2_{\omega^t}-||a(s)||^2_{\omega^t}.
\end{equation*}
Since $a(s), \ b(s)$ and $c(s)$ only depend on $u_j^{(s)}|_{D_{t_0}}$, by Lemma \ref{le:extension}, let $s\to 0$, we know that \eqref{eq:final-CF} is true at $t_0$. Since $t_0$ is an arbitrary point in $\mathbb B$, the proof of Theorem \ref{th:CF} is complete.

\subsection{Proof of Corollary \ref{co:nakano}}\ \

\medskip

For any fixed $t_0\in\mathbb B$, one may choose a sufficiently large positive constant $A$ such that $\rho+A|t|^2$ is strictly plurisubharmonic in a neighborhood of the closure of $D\cap\pi^{-1}(U)$, where $U$ is a small neighborhhod of $t_0$. Now for every $\varepsilon >0$, 
\begin{equation*}
    h^\varepsilon:=he^{-\varepsilon(\rho+A|t|^2)},
\end{equation*}
defines a smooth Hermitian metric on $\mathcal L$ with positive curvature on a neighborhood of the closure of $D\cap\pi^{-1}(U)$. Denote by $\mathcal H^\varepsilon$ the associatied family of Hilbert spaces with respect to $h^\varepsilon$. Denote by $\Theta_{j\bar k}^{\varepsilon}$ the assocaited curvature operator on $\mathcal H^\varepsilon$.  Since the total space $D$ is Stein, we know that $\theta_{j\bar k}$ is semi-positive. By the construction of $h^\varepsilon$, we know that $c_{j\bar k}(h^\varepsilon)$ is positive on $D\cap\pi^{-1}(U)$. Thus our main theorem implies that $\mathcal H^\varepsilon$ is Nakano positive on $U$. By Hamilton's theorem, we have
\begin{equation*}
    \sum(\Theta_{j\bar k} u_j, u_k)(t_0)= \lim_{\varepsilon\to 0} \sum (\Theta_{j\bar k}^{\varepsilon}u_j, u_k)(t_0)\geq 0, \ \forall \ u,v\in\Gamma(\mathcal H).
\end{equation*}
Thus $\mathcal H$ is Nakano semi-positive at $t_0$. Since $t_0$ is an arbitrary point in $\mathbb B$, we know that $\mathcal H$ is Nakano semi-positive.

\section{Curvature of the dual family}

In this section, we shall prove our main application Corollary \ref{co:dual-psh}. As a direct application, we shall give a plurisubharmonicity property of the derivatives of the Bergman kernel, which can be seen as a generalization of Theorem C. In the last part of this section, based on a remarkable idea of Berndtsson and Lempert \cite{BL14}, we shall show how to use Corollary \ref{co:dual-psh} to study plurisubharmonicity properties of the Bergman projection of currents with compact support.

\subsection{Proof of Corollary \ref{co:dual-psh}}\ \

\medskip

Let $f$ be a holomorphic section of the dual of $\mathcal H$. By Definition \ref{de:smooth-dual-new}, we know that there is a smooth section, say $P(f)$, of $\mathcal H$, such that
\begin{equation}
f^t(u^t)=(u^t, P(f)^t),
\end{equation}
for every $u^t\in \mathcal H_t$. Moreover, by Definition \ref{de:holomorphic-dual-new}, we know that
\begin{equation}
f(u): t\mapsto f^t(u^t),
\end{equation}
is a holomorphic function of $t$ if $u$ is a holomorphic section of $\mathcal H$. Thus we have
\begin{equation}
0\equiv \dbar_{t^j}f(u) = (u, D_{t^j}P(f)),
\end{equation}
for every holomorphic section $u$ of $\mathcal H$. By Lemma \ref{le:extension}, we know that 
\begin{equation}\label{eq:new-important}
D_{t^j}P(f) \equiv 0,
\end{equation}
which implies that
\begin{equation}
\partial_{t^j}\dbar_{t^k}(||P(f)||^2)=(\dbar_{t^k}P(f), \dbar_{t^j}P(f))+(\Theta_{j\bar k}P(f), P(f)).
\end{equation}
By Corollary \ref{co:nakano}, we have  
\begin{equation*}
\sum (\Theta_{j\bar k}(\xi_j P(f)),\xi_k P(f)) 
\geq 0.
\end{equation*}
for every $\xi\in\C^m$. Thus we have
\begin{equation*}
    \sum \dbar_{t^k}\partial_{t^j}(\log||P(f)||^2) \xi_j\bar\xi_k \geq 
    \frac{||\sum\bar\xi_k \dbar_{t^k}P(f)||^2}{||P(f)||^2}-
    \frac{|(P(f), \sum\bar\xi_k \dbar_{t^k}P(f) )|^2}{||P(f)||^4},
\end{equation*}
on
\begin{equation}
U_f:= \{t\in\mathbb B: ||P(f)^t||>0\}.
\end{equation}
By Schwartz inequality, we have $\sum \dbar_{t^k}\partial_{t^j}(\log||P(f)||^2) \xi_j\bar\xi_k \geq 0$ on $U_f$. Notice that
\begin{equation}
||P(f)||:  t\mapsto ||P(f)^t||=||f^t||
\end{equation}
is a smooth function on $\mathbb B$. Thus $\log||P(f)||=\log||f||$ is plurisubharmonic on $\mathbb B$. The proof is complete.

\subsection{Variation of the derivatives of the Bergman kernel}\ \

\medskip

For simplicity purposes, we shall only consider the following case:  

\medskip

\textbf{Pseudoconvex family in $\C^n$}: In this case, $\mathcal X$ is $\mathbb C^n \times \mathbb B$ and $\pi$ is just the natural projection to $\mathbb B$. Assume that $D$ satisfies $\mathbf{A1}$ and $\mathbf{A2}$. One may look at $D=\{D_t\}_{t\in\mathbb B}$ as a smooth family of smoothly bounded strongly pseudoconvex domains in $\C^n$. Moreover, we shall assume that $\mathcal L$ is a trivial line bundle over $\mathcal X$ with Hermitian metric $h=e^{-\phi}$. 

\medskip

\textbf{Variation formula of the derivatives of the Bergman kernel}: Fix $\eta\in D_0$, replace $\mathbb B$ by a smaller ball if necessary, one may assume that 
\begin{equation}
\eta\in D_t, \ \ \forall\ t\in\mathbb B.
\end{equation}
Let us consider
\begin{equation*}
 {D^\alpha\eta}: f\mapsto f_\alpha (\eta):=\frac{\partial^{|\alpha|}f}{(\partial\mu^1)^{\alpha_1}\cdots(\partial\mu^n)^{\alpha_n}}(\eta), \ \ \forall \ f=f(\mu)d\mu\in \mathcal H_t,
\end{equation*}
where $\alpha\in\mathbb N^n, \ |\alpha|:=\alpha_1+\cdots+\alpha_n$ and $d\mu$ is short for $d\mu^1\wedge \cdots\wedge d\mu^n$. By Definition \ref{de:holomorphic-dual-new}, we know that every ${D^\alpha\eta}$ defines a holomorphic section of the dual of $\mathcal H$. Put
\begin{equation*}
    {\cdot_\alpha\eta}:=P(D^\alpha\eta).
\end{equation*}
i.e., ${\cdot_\alpha\eta}$ is the unique smooth section of $\mathcal H$ such that
\begin{equation}\label{eq:def-point}
 ( f, {\cdot_\alpha\eta} )=f_\alpha (\eta) \ \ \forall \ f\in \mathcal H_t.
\end{equation}
Let $K^t(\zeta,\eta)d\zeta\otimes\overline{d\eta}$ be the Bergman reproducing kernel of $\mathcal H_t$. Then \eqref{eq:def-point} implies that
\begin{equation}\label{eq:reproducing-kernel}
  \cdot_0\eta=K^t(\mu,\eta)d\mu, \ ( {\cdot_\beta\eta}, {\cdot_\alpha\zeta} )
  =({\cdot_\beta\eta})_\alpha(\zeta)=\overline{({\cdot_\alpha\zeta})_\beta(\eta)}=K^t_{\alpha\bar\beta}(\zeta,\eta),
\end{equation}
where
\begin{equation*}
    K^t_{\alpha\bar\beta}(\zeta,\eta)=
    \frac{\partial^{|\alpha|+|\beta|}K^t}{(\partial\zeta^1)^{\alpha_1}\cdots(\partial\zeta^n)^{\alpha_n}
    (\partial\bar\eta^1)^{\beta_1}\cdots(\partial\bar\eta^n)^{\beta_n}}(\zeta,\eta).
\end{equation*}
By \eqref{eq:new-important}, we have
\begin{equation}\label{eq:vf-pre}
    K^t_{j\bar k\alpha\bar\beta}(\zeta,\eta)=\partial_{t^j}\dbar_{t^k}({\cdot_\beta\eta}, {\cdot_\alpha\zeta})
    =( ({\cdot_\beta\eta})_{\bar k}, ({\cdot_\alpha\zeta})_{\bar j})
    + (\Theta_{j\bar k}({\cdot_\beta\eta}), {\cdot_\alpha\zeta}).
\end{equation}
By \eqref{eq:final-CF}, we have
\begin{equation}\label{eq:vf-curvature}
    (\Theta_{j\bar k}({\cdot_\beta\eta}), {\cdot_\alpha\zeta})=\int_{\partial D_t} \theta_{j\bar k}(\rho) \langle {\cdot_\beta\eta},{\cdot_\alpha\zeta}\rangle d \sigma
    + (c_{j\bar k}(h) {\cdot_\beta\eta}, {\cdot_\alpha\zeta}) + R,
\end{equation}
where
\begin{equation}\label{eq:vf-R}
    R=(c,c')_{i\partial\dbar\phi|_{D_t}}+(b,b')_{\omega^t}-(a,a')_{\omega^t},
\end{equation}
if $i\partial\dbar\phi|_{D_t}>0$, and
\begin{equation*}
    R=(b,b')_{\omega^t}-(a,a')_{\omega^t},
\end{equation*}
if $i\partial\dbar\phi\equiv 0$. Here
\begin{equation*}
    \omega^t=i\partial\dbar(-\log-\rho)|_{D_t},
\end{equation*}
and $(a,b,c)$ (resp. $(a',b',c')$) are forms associated to ${\cdot_\beta\eta}$ (resp. ${\cdot_\alpha\zeta}$) respectively. Moreover,
\begin{equation*}
    \dbar^t a=\partial^t_\phi b+c, \ \dbar^t a'=\partial^t_\phi b'+c'.
\end{equation*}

\textbf{Remark:} Theorem \ref{th:L2} implies that $R$ is non-negative as a Hermitian form. Later we shall give
an explicit expression of the H\"ormander remaining term $R$ in case $i\partial\dbar\phi\equiv 0$. (thus $c=c'\equiv 0$). 

\medskip

\textbf{H\"ormander remaining term for flat weight}: Let us assume that $i\partial\dbar\phi\equiv 0$. By definition, then we have $c=c'\equiv 0$. Put
\begin{equation}
\square'=\partial^t_{\phi}(\partial^t_{\phi})^*+(\partial^t_{\phi})^*\partial^t_{\phi},
\end{equation}
We shall prove that:

\begin{lemma}\label{le:R-expression} If $i\partial\dbar\phi\equiv 0$ on $D$ then
\begin{equation}\label{eq:R-expression}
    R=(\mathbb H b,\mathbb H b')_{\omega^t},
\end{equation}
where $\mathbb H b$ denotes the $\square'$-harmonic part of $b$.
\end{lemma}

\begin{proof} Since $i\partial\dbar\phi\equiv 0$ and $\omega^t$ is complete K\"ahler, we know that the $\dbar$-Laplace $\square''$ is equal to $\square'$. Denote by $G$ the associated Green operator. Let us omit $\omega^t$ in $(\cdot,\cdot)_{\omega^t}$, then we have
\begin{equation*}
    (a,a')=((\dbar^t)^*G\partial^t_\phi b, a')=(G\partial^t_\phi b, \partial^t_\phi b').
\end{equation*}
Since $b$ is  primitive and $\dbar^t$-closed, we know that $b$ is $(\partial^t_\phi)^*$-closed. Thus $b$ can be written as
\begin{equation*}
    b=\mathbb H b + (\partial^t_\phi)^*f, \ \ \partial^t_\phi f=0.
\end{equation*}
Now
\begin{equation*}
    (a,a')=(G\partial^t_\phi(\partial^t_\phi)^*f, \partial^t_\phi b')=(f,\partial^t_\phi b')=(b-\mathbb H b,b')=(b,b')-(\mathbb H b, \mathbb H b').
\end{equation*}
Thus
\begin{equation*}
    R=(b,b')-(a,a')=(\mathbb H b, \mathbb H b').
\end{equation*}
\end{proof}

Recall that 
\begin{equation}
b=(\dbar V_j)|_{D_t} ~ \lrcorner ~ ({\cdot_\beta\eta}), \ b'=(\dbar V_k)|_{D_t} ~ \lrcorner ~ ({\cdot_\alpha\zeta}).
\end{equation}
Thus Lemma \ref{le:R-expression} implies that:

\begin{theorem}[Variation Formula of the Bergman Kernel]\label{th:VF} The first order variation formula of the Bergman kernel can be written as
\begin{equation}\label{eq:VF1}
    K^t_{j\alpha\bar\beta}(\zeta,\eta)=i^{n^2}\int_{D_t}  \phi_j\{{\cdot_\beta\eta},{\cdot_\alpha\zeta}\}
    -i^{n^2}\int_{\partial D_t} \delta_{V_j}\{{\cdot_\beta\eta},{\cdot_\alpha\zeta}\},
\end{equation}
Moreover, if $i\partial\dbar\phi\equiv 0$ on $D$ then
\begin{equation}\label{eq:VF2}
    K^t_{j\bar k \alpha\bar\beta}(\zeta,\eta)
    =( ({\cdot_\beta\eta})_{\bar k}, ({\cdot_\alpha\zeta})_{\bar j})
    +\int_{\partial D_t} \theta_{j\bar k}(\rho) \langle {\cdot_\beta\eta}, {\cdot_\alpha\zeta}\rangle d \sigma
 + (\mathbb H b,\mathbb H b'),
\end{equation}
where $b=(\dbar V_j)|_{D_t} ~ \lrcorner ~ ({\cdot_\beta\eta}), \ b'=(\dbar V_k)|_{D_t} ~ \lrcorner ~ ({\cdot_\alpha\zeta})$.
\end{theorem}

\begin{proof} Notice that \eqref{eq:VF2} is a direct consequence of Lemma \ref{le:R-expression}. Thus it suffices to prove \eqref{eq:VF1}. Notice that \eqref{eq:Lie-first} implies that
\begin{equation*}
K^t_{j\alpha\bar\beta}(\zeta,\eta)=i^{n^2} \int_{D_t} L^t_{V_j}\{ {\cdot_\beta\eta}, {\cdot_\alpha\zeta}\}.
\end{equation*}
By Cartan's formula
\begin{equation*}
    L^t_{V_j}=i_t^*(d\delta_{V_j}+\delta_{V_j}d),
\end{equation*}
thus we have
\begin{equation*}
    K^t_{j\alpha\bar\beta}(\zeta,\eta)=i^{n^2}\int_{\partial D_t} \delta_{V_j}\{{\cdot_\beta\eta},{\cdot_\alpha\zeta}\}+
    i^{n^2}\int_{D_t}  \frac{\partial}{\partial t^j}\{{\cdot_\beta\eta},{\cdot_\alpha\zeta}\}.
\end{equation*}
By the reproducing formula,
\begin{equation*}
     i^{n^2}\int_{D_t}  \frac{\partial}{\partial t^j}\{{\cdot_\beta\eta},{\cdot_\alpha\zeta}\}=2K^t_{j\alpha\bar\beta}(\zeta,\eta)-i^{n^2}\int_{D_t}  \phi_j\{{\cdot_\beta\eta},{\cdot_\alpha\zeta}\},
\end{equation*}
which implies \eqref{eq:VF1}.
\end{proof}

\textbf{Remark}: If $\alpha=\beta=0$ and $\phi\equiv0$ then \eqref{eq:VF1} is Komatsu's formula (see \cite{Komatsu82}). Recently, Berndtsson \cite{Bern15} showed that \eqref{eq:VF1} can be used to study the comparison principle for Bergman kernels. In fact, if $D$ is a product then \eqref{eq:VF1} is just (2.2) in \cite{Bern15}.

\subsection{Variation of the Bergman projection of currents}\ \

\medskip

In the last section, we discussed the plurisubharmonicity properties of the Bergman projection of the derivatives of the Dirac measure. Recently, it is known that the plurisubharmonicity properties of the Bergman projection of other kind of currents are also very useful (see \cite{BL14}). In this subsection, we shall show how to use Corollary \ref{co:dual-psh} to study variation of the Bergman projection of general currents with compact support.

\medskip

\textbf{Smooth family of currents with compact support}: Denote by $A_t$ the space of smooth sections of $K_{D_t}+L_t$ over $D_t$. Put
\begin{equation*}
    \mathcal A =\{A_t\}_{t\in\mathbb B}.
\end{equation*}
We shall introduce the notion of the dual of $\mathcal A$ by using the language of currents. Denote by $A'_t$ the dual space of $A_t$, that is the space of $L_t^*$-valued degree $(0,n)$-currents with \textbf{compact support} in $D_t$. Fix $f^t\in A'_t$, we shall \emph{formally} write
\begin{equation*}
    f^t(u^t)= \int_{D_t} f^t\wedge u^t,\ \forall \ u^t\in A_t,
\end{equation*}
even though the $(n,n)$-current $ f^t\wedge u^t$ may not be integrable in general. Put
\begin{equation*}
    \mathcal A'=\{A'_t\}_{t\in\mathbb B}.
\end{equation*}
Denote by ${\rm Supp} f^t$ the support of $f^t$. Denote by $K_{\mathcal X/\mathbb B}$ the relative canonical line bundle associated to $\pi$, recall that
\begin{equation}\label{eq:relative-canonical}
K_{\mathcal X/\mathbb B}:=K_{\mathcal X}-\pi^*K_{\mathbb B}, \ K_{\mathcal X/\mathbb B}|_{D_t} \simeq K_{D_t}.
\end{equation}
We shall introduce the following definiton:

\begin{definition} We call $f: t\to f^t\in A'_t$ a smooth family of currents with compact support if
\begin{equation}\label{eq:smooth-dual1}
    \bigcup_{t\in K} {\rm Supp} f^t  \Subset D, \ \forall \ K\Subset\mathbb B,
\end{equation}
and for every smooth section $\kappa$ of $(K_{\mathcal X/\mathbb B}+\mathcal L)\boxtimes (\overline{K_{\mathcal X/\mathbb B}}+\mathcal L^*) $ over
\begin{equation*}
\mathcal X\times_{\pi} \mathcal X:= \{(x,y)\in\mathcal X\times \mathcal X : \pi(x)=\pi(y)\},
\end{equation*}
there exists a smooth section, say $u_{f,\kappa}$, of $\overline{K_{\mathcal X/\mathbb B}}+\mathcal L^*$ over $\mathcal X$ such that
\begin{equation}\label{eq:smooth-dual2}
    f^t(\kappa^t(v^t))=u_{f,\kappa}^t(v^t), \ \forall \ v\in C^\infty(\mathcal X, K_{\mathcal X/\mathbb B}+\mathcal L), \ t\in \mathbb B.
\end{equation}
\end{definition}

\textbf{Remark}: Let us explain the meaning of \eqref{eq:smooth-dual2}. The right hand side is clear, that is
\begin{equation*}
    u_{f,\kappa}^t(v^t):=\int_{D_t} u_{f,\kappa}^t\wedge v^t.
\end{equation*}
For the left hand side, by our assumption \textbf{A1}, the restriction of $\pi$ to the closure of $D$ is proper, thus we know that
\begin{equation*}
    \kappa^t(v^t): x \mapsto \int_{D_t} \kappa^t(x,\cdot)\wedge v^t(\cdot), \ \forall \ x\in \pi^{-1}(t),
\end{equation*}
defines a section in $A_t$. Thus $f^t(\kappa^t(v^t))$ is well defined. Hence \eqref{eq:smooth-dual2} means that \emph{the current defined by $f(\kappa)$ is smooth up to the boundary of $D$}. 

\medskip

\textbf{Bergman projection of smooth family of currents with compact support}: We shall prove the following proposition:

\begin{proposition} Assume that $D$ satisfies $\mathbf{A1}$ and $\mathbf{A2}$.  Let $f: t\to f^t\in A'_t$ be a smooth family of currents with compact support. Then
\begin{equation}
f^t: u^t \mapsto f^t(u^t), \ \ \forall \ u^t\in\mathcal H_t,
\end{equation}
defines a smooth section of the dual of $\mathcal H$ in the sense of Definiton \ref{de:smooth-dual-new}.
\end{proposition}

\begin{proof} \eqref{eq:smooth-dual1} implies that there exists a smooth real function, say $\chi$, on $D$ such that
\begin{equation*}
    \chi \equiv 1 \ \text{on} \ \bigcup_{t\in\mathbb B} {\rm Supp} f^t,  \ \ {\rm Supp} ( \chi|_{D_t} )\Subset D_t, \ \forall\ t\in\mathbb B.
\end{equation*}
Denote by $K^t$ the Bergman kernel of $\mathcal H_t$. Put
\begin{equation*}
    \chi K: (x,y) \mapsto \chi(x) K^{\pi(x)}(x,y), \ \forall \ (x,y)\in D\times_{\pi}D.
\end{equation*}
By Hamilton's theorem (see Appendix \ref{ss:SBK}), assumptions \textbf{A1} and \textbf{A2} imply that $\chi K$ is smooth up to the boundary, i.e., it extends to a smooth section of $(K_{\mathcal X/\mathbb B}+\mathcal L)\boxtimes (\overline{K_{\mathcal X/\mathbb B}}+\mathcal L^*) $ over
$\mathcal X\times_{\pi} \mathcal X$. By the reproducing property of $K^t$, we have
\begin{equation*}
    (\chi K)^t(v^t)=(\chi v)^t, \ \forall\ v\in \Gamma(\mathcal H).
\end{equation*}
Thus by \eqref{eq:smooth-dual2}, we have
\begin{equation}\label{eq:fv??}
    f^t(v^t)=f^t((\chi v)^t)=f^t((\chi K)^t(v^t))=u^t_{f, \chi K}(v^t)=\int_{D_t} u^t_{f, \chi K} \wedge v^t, \ \forall \ v^t\in\mathcal H_t.
\end{equation}
Let us write
\begin{equation*}
    u^t_{f, \chi K} \wedge v^t =i^{n^2} \{v^t,P(f)^t\}, \ \forall\ v^t\in C^{\infty}(D_t, K_{D_t}+L_t).
\end{equation*}
Thus we have
\begin{equation}\label{eq:formal-good}
    f^t((\chi K)^t(v^t))=(v^t, P(f)^t), \ \forall\ v^t\in C^{\infty}(D_t, K_{D_t}+L_t).
\end{equation}
Since $(\chi K)^t(\mathcal H_t^{\bot})=0$, we have $P(f)^t\in \mathcal H_t$. Thus $P(f)\in\Gamma(\mathcal H)$, and by \eqref{eq:fv??}, we have
\begin{equation}\label{eq:formal-good-1}
    f^t(v^t))=(v^t, P(f)^t), \ \forall \ v\in \Gamma(\mathcal H).
\end{equation}
By Definition \ref{de:smooth-dual-new}, we know that $f$ defines a smooth section of the dual of $\mathcal H$.
\end{proof}

\textbf{Remark} By Lemma \ref{le:extension}, if $D$ satisfies $\mathbf{A1}$ and $\mathbf{A2}$ then $\mathcal H_t$ is equal to the closure of $\{u^t\in\mathcal H_t: u\in\Gamma(\mathcal H)\}$. Thus \eqref{eq:formal-good-1} implies that
\begin{equation}\label{eq:norm-dual}
||P(f)||^2(t)= \sup\{|f^t(u)|^2: u\in \mathcal H_t , \ i^{n^2}\int_{D_t}\{u,u\}=1\}.
\end{equation}
By this extremal property, one may generalize Corollary \ref{co:dual-psh} to the case that the metric $h$ on $\mathcal L$ is singular. 

\section{Triviality and flatness}

In this section, we shall prove Theorem \ref{th:flat} and use it to study triviality of holomorphic motions.

\subsection{Proof of Theorem \ref{th:flat}}\ \

\medskip

\textbf{Triviality implies flatness}: By definition, if $D$ is trivial then one may assume that the vector fields $\partial/\partial t^j$ are well defined on $D$, tangent to the boundary of $D$ and can be extended to smooth vector fields on $\mathcal X$. Thus we have
\begin{equation}
\theta_{j\bar k}(\rho)\equiv 0,
\end{equation}
on $\partial D$. Moreover, in this case $b\equiv 0$. If $\Theta(\mathcal L, h)\equiv 0$ then we also have $c\equiv 0$. Thus $a\equiv 0$ and $R\equiv 0$. By our main theorem, we know that $\mathcal H$ is flat.

\medskip

\textbf{Flatness implies triviality}: By Theorem \ref{th:CF} and our assumption, we have
\begin{equation*}
    \sum(\Theta_{j\bar k} u_j, u_k)= \sum\int_{\partial D_t} \theta_{j\bar k}(\rho) \langle u_j,u_k\rangle d \sigma + R,
\end{equation*}
where $R\geq 0$. Moreover, since $D$ is Stein, we have
\begin{equation*}
    \sum\int_{\partial D_t} \theta_{j\bar k}(\rho) \langle u_j,u_k\rangle d \sigma \geq 0.
\end{equation*}
Thus if $\Theta_{j\bar k}\equiv 0$ then $R\equiv 0$ and
\begin{equation}\label{eq:-log-rho-1}
   \theta_{j\bar k}(\rho)=\langle V_j^{\rho}, V_k^{\rho} \rangle_{i\partial\op\rho}\equiv 0 \ \text{on} \ \partial D.
\end{equation}
Since
\begin{equation}\label{eq:-log-rho}
    \langle V_j, V_k \rangle_{i\partial\op(-\log-\rho)}=\frac{\langle V_j, V_k \rangle_{i\partial\op\rho}}{-\rho}+
                 \frac{V_j(\rho) \overline{V_k(\rho)}}{\rho^2},
\end{equation}
and $V_j=V_j^\rho$ on $\partial D$, by \eqref{eq:-log-rho} and \eqref{eq:-log-rho-1}, we know that $\langle V_j, V_k \rangle_{i\partial\op(-\log-\rho)}$ is smooth up to the boundary of $D$. We shall use the following lemma, which follows from
\begin{equation}
V_j^\psi=\partial/\partial t^j-\sum \psi_{j\bar\lambda}\psi^{\bar\lambda\nu}\partial/\partial \mu^\nu,
\end{equation}
by direct computation.

\begin{lemma}\label{le:geodesic} Let $\psi$ be a smooth function on $D$. Assume that $\psi$ is strictly plurisubharmonic on each fibre of $D$. Denote by $V_j^\psi$ the horizontal lift of $\partial/\partial t^j$ with respect to $i\partial\dbar\psi$. Then
\begin{equation*}
    [V_j^\psi, V_k^\psi]=0,  \ \ [V_j^\psi, \overline{V_k^\psi}]=\sum c_{j\bar k}(\psi)_{\bar\lambda}\psi^{\bar\lambda\nu}\partial/\partial\mu^\nu-c_{j\bar k}(\psi)_\nu\psi^{\bar\lambda\nu}\partial/\partial\bar\mu^{\lambda},
\end{equation*}
where $c_{j\bar k}(\psi):=\langle V_j^\psi, V_k^\psi \rangle_{i\partial\dbar\psi}$.
\end{lemma}

Let us apply this lemma to $\psi=-\log-\rho$. Now by definition of $V_j$ in Lemma \ref{le:key-lemma}, we have $V_j \equiv V_j^\psi$. Since $(-\log-\rho)^{\bar\lambda\nu} \equiv 0$ on $\partial D$. By \eqref{eq:-log-rho} and the above lemma, we have
\begin{equation}\label{eq:flat-bdy}
    [V_j, V_k]=0,  \ \text{on} \ D, \ \text{and} \ \ [V_j, \overline{V_k}]=0,\ \text{on}  \ \partial D.
\end{equation}
on the boundary of $D$. Moreover, we shall prove that the following lemma is true.

\begin{lemma}\label{le:equality1} Assume that $D$ satisfies $\mathbf{A1}$ and $\mathbf{A2}$. Assume further that $K_{\mathcal X/\mathbb B} +\mathcal L$ is trivial on each fibre of $\pi$ and $\Theta(\mathcal L, h)\equiv 0$ on $D$. If $R\equiv0$ then each $V_j^\rho$ has a smooth extension, say $\tilde V_j$, that is holomorphic on fibres and smooth up to the boundary of $D$.
\end{lemma}

If the above lemma is true then by \eqref{eq:flat-bdy}, we have
\begin{equation*}
    [ \tilde V_j,\tilde V_k]=[\tilde V_j,\overline{ \tilde V_k}]=0, \ \ \text{on}\ \partial D.
\end{equation*}
Since $\tilde V_j$ are holomorphic on fibres, we have
\begin{equation*}
    [\tilde V_j,\overline{ \tilde V_k}] =  \frac{\partial}{\partial  t^j} \overline{ \tilde V_k} - \frac{\partial}{\partial \bar t^k} \tilde V_j.
\end{equation*}
Thus
\begin{equation*}
    \frac{\partial}{\partial \bar t^k} \tilde V_j \equiv 0, \ \text{on}\ \partial D.
\end{equation*}
Since $\frac{\partial}{\partial \bar t^k} \tilde V_j$ are holomorphic on each fibre, we have
\begin{equation*}
    \frac{\partial}{\partial \bar t^k} \tilde V_j \equiv 0, \ \text{on}\  D.
\end{equation*}
Thus each $\tilde V_j$ is a holomorphic vector field on $D$. Moreover,
\begin{equation*}
    [\tilde V_j,\tilde V_k]\equiv 0,\ \text{on}\ \partial D.
\end{equation*}
Thus $D$ is trivial. The proof of Theorem \ref{th:flat} is complete.

\medskip

Now let us prove Lemma \ref{le:equality1}:

\medskip

\textbf{Proof of  Lemma \ref{le:equality1}}: We shall prove that $R\equiv 0$ implies that every $V_j|_{\partial D}(=V_j^\rho)$ has a holomorphic extension to $D$. Notice that the proof of Lemma \ref{le:R-expression} implies that
\begin{equation*}
    R=||\mathbb H (\sum(\dbar V_j)|_{D_t} ~ \lrcorner~ u_j)||^2_{\omega^t},
\end{equation*}
where $\omega^t=i\partial\dbar(-\log-\rho)|_{D_t}$. Thus $R\equiv 0$ implies that $\mathbb H b \equiv 0$. Since $\omega^t$ is $d$-bounded in the sense of Gromov (see \cite{Gromov91}, \cite{DF83} or \cite{Bern96}), and
\begin{equation*}
b:=\sum(\dbar V_j)|_{D_t} ~ \lrcorner~ u_j
\end{equation*}
is $\dbar$-closed, we know that there exists a smooth $L_t$-valued $(n-1,0)$-form $u^t$ such that $\dbar^t u^t=b$ and
\begin{equation*}
    ||u^t||_{\omega^t} \leq 2 ||b||_{\omega^t}<\infty.
\end{equation*}
We claim that $u: (\eta, t)\mapsto u^t(\eta)$ is smooth up the boundary of $D$ and $u=0$ on $\partial D$. In fact, if we can show
\begin{equation}\label{eq:equality1}
    \int_{D_t} \{f,b\}=0,
\end{equation}
for every $\partial^t_\phi$-closed $L_t$-valued $(n-1,1)$-form $f$ that is smooth up to the boundary of $D_t$, then $*b\in {\rm Im} (\partial^t_\phi)^*$, where $*$ is the Hodge-de Rham operator with respect to $(i\partial\dbar \rho)|_{D_t}$ and $(\partial^t_\phi)^*$ is the adjoint of $\partial^t_\phi$ with respect to $(i\partial\dbar \rho)|_{D_t}$. By the regularity property of the $\dbar$-Neumann problem (in fact, in our case, it is Dirichlet problem), one may solve
\begin{equation*}
    (\partial^t_\phi)^*v^t=*b.
\end{equation*}
where $v: (\eta,t) \mapsto v^t(\eta)$ is smooth up to the boundary of $D$. Since $(\partial_\phi^t)^*=-*\dbar^t*$, we have
\begin{equation*}
 \dbar^t(-*v^t)=b.
\end{equation*}
Since $v^t\in {\rm Dom} (\partial_\phi^t)^*$, we have $v^t=0$ on $\partial D_t$. Thus $||-*v^t||_{\omega^t}<\infty$. Since there are no $L^2$ (with respect to $\omega^t$) holomorphic $L_t$-valued $(n-1,0)$-forms on $D_t$, we have $u^t=-*v^t$. Thus our claim follows from \eqref{eq:equality1}.

Now let us prove \eqref{eq:equality1}. Put $\rho^t=\rho|_{D_t}$. Since $b$ is smooth up to the boundary, we have
\begin{equation}\label{eq:equality2}
    \int_{D_t} \{f, b\}=\lim_{r\to 0-}\int_{\{\rho^t<r\}}\{f, b\}= \lim_{r\to 0-} (-1)^n
    \int_{\{\rho^t=r\}}\{f,u^t\}.
\end{equation}
Since
\begin{equation*}
    \omega^t\geq \frac{(i\partial\dbar\rho)|_{D_t}}{-\rho},
\end{equation*}
we know that $||u^t||_{\omega^t}<\infty$ implies that
\begin{equation*}
    \int_{D_t}\frac{|u^t|_{i\partial\dbar\rho|_{D_t}}^2}{-\rho} \frac{(i\partial\dbar\rho)^n|_{D_t}}{n!} <\infty.
\end{equation*}
Thus
\begin{equation*}
    \liminf_{r\to 0-}\int_{\{\rho^t=r\}} |u^t|^2_{i\partial\dbar\rho|_{D_t}} d\sigma = 0.
\end{equation*}
Since $f$ is smooth up to the boundary, we know that 
\begin{equation}\label{eq:equality3}
    \lim_{r\to 0-} (-1)^n
    \int_{\{\rho^t=r\}}f\wedge \bar u^t=0.
\end{equation}
Thus \eqref{eq:equality1} follows from \eqref{eq:equality2} and \eqref{eq:equality3}, and our claim is proved.

By our assumption, $K_{\mathcal X/\mathbb B} +\mathcal L$ is trivial on $\pi^{-1}(t)$, thus there exists a holomorphic section, say $e$, of $K_{\mathcal X/\mathbb B} +\mathcal L$, that has no zero point in $\pi^{-1}(t)$. Now fix $t\in\mathbb B$, $1\leq j\leq m$. Put
\begin{equation*}
    u_j=e, \ \ u_k=0, \ \forall \ k\neq j.
\end{equation*}
By our claim, one may solve $\dbar^t u^t=(\dbar V_j)|_{D_t} ~\lrcorner~ e$ such that $u^t=0$ on the boundary. Since $e$ has no zero point in $\pi^{-1}(t)$, one may write
\begin{equation*}
    u^t=V ~\lrcorner~ e, \ \text{on}\ D_t.
\end{equation*}
Thus we have
\begin{equation*}
    \dbar^t(V_j-V)=0,  \ \text{on}\ D_t; \ V_j-V=V_j \ \text{on}\ \partial D_t.
\end{equation*}
Thus $V_j|_{\partial D_t}$ has a holomorphic extension, say $\tilde{V_j}|_{D_t}$, to $D_t$. The regularity property of $\tilde{V_j}$ follows from the regularity property of $u^t$.
 The proof is complete.

\subsection{Triviality of holomorphic motions}\ \

\medskip

We shall show how to use Theorem \ref{th:flat} to study triviality of holomorphic motions.

\medskip

\textbf{Basic notions on holomorphic motion}: Recall that  that if every fibre $D_t$ is a domain in $\C$ then:

\medskip

\emph{ $\partial D$ is Levi-flat if and only if $\theta_{j\bar k}(\rho)\equiv 0$.} 

\medskip

It is known that the boundary of the total space of a holomorphic motion of a planar domain is Levi-flat. Recall that, by definition, a homeomorphism 
\begin{equation}
F:(z,t)\mapsto(f(z,t),t),
\end{equation}
from $D_0\times\mathbb B$ to $D$ is called a \emph{holomorphic motion} (see \cite{MSS83}) of $D_0$ (with total space $D$) if $f(\cdot,0)$ is the identity mapping and $f(z,\cdot)$ is holomorphic for every fixed $z\in D_0$. 

\medskip

\textbf{Curvature formula for holomorphic motions}: Assume that $D_0$ is a smooth domain in $\C$ and $F$ is smooth up to the boundary. Assume further that $\mathcal L$ is trivial and $\phi\equiv 0$ on $D$. Put
\begin{equation*}
    V_j^F:=F_*(\partial/\partial t^j)=\partial/\partial t^j+f_j(z,t)\partial/\partial \zeta.
\end{equation*}
By definition, $V_j^F(\rho)\equiv0$ on $\partial D$. Thus
\begin{equation*}
    V_j^F=V_j=V_j^\rho, \ \text{on}\ \partial D.
\end{equation*}
Hence we have
\begin{equation*}
    ||\sum (V^F_j-V_j)~\lrcorner~ u_j||_{\omega^t}<\infty,
\end{equation*}
which implies that
\begin{equation*}
    \sum(\Theta_{j\bar k} u_j, u_k)=||\mathbb H (\sum ( \dbar V_j)|_{D_t}~\lrcorner~ u_j)||^2=||\mathbb H (\sum ( \dbar V_j^F)|_{D_t}~\lrcorner~ u_j)||^2.
\end{equation*}

\textbf{Criterion for $\Theta_{j\bar k}\equiv 0$ by using the Bergman kernel}: Put
\begin{equation*}
J=f_{\bar z}/f_z.
\end{equation*}
Since 
\begin{equation*}
\partial/\partial \bar\zeta=z_{\bar\zeta}\partial/\partial z+ \overline{z_{\zeta}} \partial/\partial \bar z,
\end{equation*}
and
\begin{equation*}
    \overline{z_{\zeta}}=\frac{f_z}{|f_z|^2-|f_{\bar z}|^2}, \ z_{\bar\zeta}=\frac{-f_{\bar z}}{|f_z|^2-|f_{\bar z}|^2},
\end{equation*}
we have
\begin{equation}\label{eq:abz}
    (\dbar V_j^F)|_{D_t}=(f_{jz}z_{\bar\zeta}+f_{j\bar z}\overline{z_{\zeta}})d\bar\zeta\otimes \frac{\partial}{\partial \zeta}=\frac{(f_z)^2J_j}{|f_z|^2(1-|J|^2)} d\bar\zeta\otimes \frac{\partial}{\partial \zeta}.
\end{equation}
Thus $\Theta_{j\bar k}\equiv 0$ is equivalent to
\begin{equation}\label{eq:trivialcondition}
    \int_{D_t} K^t(\zeta,\bar\eta) \left(\frac{(f_z)^2J_j}{|f_z|^2(1-|J|^2)}\right)(z(\zeta,t),t)\ id\zeta\wedge d\bar\zeta=0.
\end{equation}
for every $(\eta,t)$ in $D$ and every $j$.

\begin{proof}[Proof of Corollary \ref{co:last}] Since $J=a(t)$ now, by \eqref{eq:trivialcondition}, we know that $\Theta_{j\bar k}\equiv 0$ is equivalent to
\begin{equation*}
 \frac{a_j}{1-|a|^2} \int_{D_t} K^t(\zeta,\bar\eta) \ id\zeta\wedge d\bar\zeta=0.
\end{equation*}
for every $(\eta,t)$ in $D$ and every $j$. But notice that
\begin{equation*}
 \int_{D_t} K^t(\zeta,\bar\eta) \ id\zeta\wedge d\bar\zeta\equiv1.
\end{equation*}
Thus $\Theta_{j\bar k}\equiv 0$ is equivalent to $a_j\equiv 0$ for every $j$. Since $a(0)=0$, we know that $\Theta_{j\bar k}\equiv 0$ is equivalent to $a\equiv 0$.
\end{proof}

\textbf{Remark}: In \cite{Liurenshan}, Ren-Shan Liu showed that if $f=z+t^2\bar z$, then $F(\mathbb D\times \mathbb D)$ is not biholomorphic equivalent to the bidisc, where $\mathbb D$ denotes the unit disc. Interested  readers can find more information on the holomorphic motion in \cite{Sl91} and \cite{ST86}.

\section{Appendix}

\subsection{Variation of fibre integrals}\label{ss:VFI} \ \

\medskip 

Let $\B$ be the unit ball in $\R^m$. Let $\{D_t\}_{t\in\B}$ be a family of smoothly bounded domains in $\R^n$. Put 
\begin{equation*}
    D:=\{(t,x)\in\R^{m+n}: x\in D_t, \ t\in\B\}.
\end{equation*}
Assume that there is a real valued function $\rho$ on $\mathbb B\times \mathbb R^n$ such that for each $t$ in $\B$, $\rho|_{D_t}$ is a smooth defining function of $D_t$. 

We call $\{D_t\}_{t\in\B}$ a smooth family if there exists a fibre preserving diffeomorphism $\Phi$ from $\mathbb B\times D_0$ onto $D$ such that for each $1\leq j\leq m$, $\Phi_*(\partial/\partial t^j)$ extends to a smooth vector field on $\mathbb R^n$. Put
\begin{equation}\label{eq:boundary}
    [D]:=\overline D\cap(\B\times\R^n), \ \delta D:=\partial D\cap(\B\times\R^n).
\end{equation}
Let $dx:=dx^1\wedge\cdots\wedge dx^n$ be the Euclidean volume form on $\R^n$. Fix a smooth function $f$ on a neighborhood of $[D]$. If $\{D_t\}_{t\in\B}$ is a smooth family then the fibre integrals
\begin{equation*}
    F(t):=\int_{D_t}f(t,x)dx
\end{equation*}
depend smoothly on $t\in \B$. We shall introduce a natural way to compute the derivatives of $F(t)$ (see \cite{Sch12} for related results). For every fixed
$j\in\{1,\cdots m\}$, let
\begin{equation*}
    V_j:=\frac{\partial}{\partial t^j}-\sum v^{\lambda}_j\frac{\partial}{\partial x^\lambda}
\end{equation*}
be a smooth vector field on a neighborhood of $[D]$. We shall prove that:

\begin{theorem}\label{th:vfi} Let $\{D_t\}_{t\in\B}$ be a smooth family of smoothly bounded domain in $\R^m$. Assume that $V_j(\rho)=0$ on $\delta D$. Then we have
\begin{equation}\label{eq:vfi}
 \frac{\partial F}{\partial t^j}(t)=\int_{D_t}L^t_{V_j}\left(f(t,x)dx\right)=\int_{D_t}L_{V_j}\left(f(t,x)dx\right),
\end{equation}
for every $t$ in $\B$, where $L^t_{V_j}:=i_t^*(L_{V_j})$.
\end{theorem}

\begin{proof} Without loss of generality, one may assume that $t=0$ and $j=1$. Since $V_1(\rho)$ vanishes on $\delta D$, the motion
\begin{equation*}
\Phi:(-1,1)\times D_0\rightarrow\R^{m}
\end{equation*}
of $D_0$ associated to $V_1$ is compatible with $\{D_t\}$, i.e.
\begin{equation*}
\Phi(a\times D_0)=D_{a\nu}, \ \nu=(1,0,\cdots,0)\in\R^m,
\end{equation*}
for every $a\in(-1,1)$. Since for every fixed $a\in(-1,1)$,
\begin{equation*}
\Phi^a:x\mapsto \Phi(a,x)
\end{equation*}
is a $C^{\infty}$ isomorphism from $D_0$ to $D_{a\nu}$, we have
\begin{equation}\label{eq:derivative}
     \frac{\partial F}{\partial t^1}(0)=\lim_{0\neq a\to 0}\int_{D_0}\frac{f(a\nu,\Phi^a(x))d\Phi^a(x)-f(0,x)dx}{a}
\end{equation}
Since $V_1$ and $f$ are smooth up to the boundary, we have
\begin{equation}\label{eq:derivative1}
     \frac{\partial F}{\partial t^1}(0)=\int_{D_0}\lim_{0\neq a\to 0}\frac{f(a\nu,\Phi^a(x))d\Phi^a(x)-f(0,x)dx}{a}.
\end{equation}
By definition of Lie derivative,
\begin{equation}\label{eq:derivative2}
L_{V_1}\left(f(t,x)dx\right)(0,x)=\lim_{0\neq a\to 0}\frac{[(\Psi^a)^*(fdx)](0,x)-f(0,x)dx}{a},
\end{equation}
where
\begin{equation*}
    \Psi^a:(b\nu,\Phi^b(x))\mapsto(b\nu+a\nu,\Phi^{b+a}(x)), \ (b,x)\in (-1+|a|,1-|a|)\times D_0.
\end{equation*}
Since
\begin{equation*}
    i_0^*\left\{[(\Psi^a)^*(fdx)](0,x)-f(a\nu,\Phi^a(x))d\Phi^a(x)\right\}=0,
\end{equation*}
\eqref{eq:vfi} follows from \eqref{eq:derivative1} and \eqref{eq:derivative2}.
\end{proof}

Now assume that $m=2$, put 
\begin{equation*}
\frac{\partial}{\partial t}:=\frac12\left(\frac{\partial}{\partial t^1}-i\frac{\partial}{\partial t^2}\right), \
\frac{\partial}{\partial \bar t}:=\frac12\left(\frac{\partial}{\partial t^1}+i\frac{\partial}{\partial t^2}\right).
\end{equation*}
Let
\begin{equation*}
V=\frac{\partial}{\partial t}-\sum v^\lambda\frac{\partial}{\partial x^\lambda}
\end{equation*}
be a smooth vector field on a neighborhood of $[D]$. If $V(\rho)$ vanishes on $\delta D$, then both $2{\rm Re}V$ and $-2{\rm Im}V$ satisfy the assumption of Theorem~\ref{th:vfi}. Thus we have:

\begin{corollary}\label{co:vfi} If $V(\rho)$ vanishes on $\delta D$ then
\begin{equation*}
 \frac{\partial F}{\partial t}(t)=\int_{D_t} L^t_{V}\left(f(t,x)dx\right), \
 \frac{\partial F}{\partial \bar t}(t)=\int_{D_t} L^t_{\overline V}\left(f(t,x)dx\right),
\end{equation*}
for every $t\in \mathbb B$.
\end{corollary}

\subsection{Stability of the Bergman kernel}\label{ss:SBK}  \ \

\medskip

We shall give a short account of Hamilton's theory on regularity properties of families of non-coercive boundary value problems. By Lemma 2.1 in \cite{Bern06}, stability of Bergman kernels follows directly from stability of solutions $u^t$ of a family of $\dbar$-Neumann problems $\square^t(\cdot)=f^t$. But it is not easy to prove regularity of $u^t$ by the standard method. In fact, if we want to use
\begin{equation}\label{eq:stability}
    ||\square^t(u^t-u^s)||=||f^t-f^s-(\square^t-\square^s)u^s||,
\end{equation}
to estimate $||u^t-u^s||$ then we have to find a natural connection between the domain of $\square^t$ and the domain of $\square^s$ (i.e., $u^s$ may not be in the domain of $\square^t$).

Hamilton \cite{Hamilton79} found a more natural way to study the regularity properties of families of non-coercive boundary value problems (not only for the $\dbar$-Neumann problem). For reader's convenience we give a sketch description of Hamilton's idea.

Instead of considering $\square^t$ (whose domain satisfies the so called $\dbar$-Neumann condition), Hamilton considered the full Laplace operator $\widetilde{\square^t}$ (whose domain contains all forms smooth up to the boundary). Let $u^t$ be a form smooth up to the boundary. In  general, the Sobolev norm of $\widetilde{\square^t}(u^t)$ could not control the Sobolev norm of $u^t$. In fact, $u^t$ has to be in the domain of $\square^t$ (see \cite{Folland-Kohn72}). Thus two more operators (sending forms on $\overline{D_t}$ to forms on the boundary of $D_t$) are used in Hamilton's paper, i.e., he considered the full $\dbar$-Neumann problem
\begin{equation}\label{eq:full-Neumann}
    \mathfrak{S}^t(\cdot):=\left(\widetilde{\square^t},(\dbar^t\rho)\vee,(\dbar^t\rho)\vee\dbar^t\right)(\cdot)=f^t,
\end{equation}
where  $$(\dbar^t\rho)\vee:=(\dbar^t\rho\wedge\cdot)^*.$$ Now the domain of $\mathfrak{S}^t$ is $C^{\infty}_{\bullet,\bullet}(\overline{D_t})$ for each $t$. Choose a $C^{\infty}$ trivialization mapping 
$$\mathbb B\times D_0\simeq D,$$
then the domain of $\mathfrak{S}^t$ can be seen as a fixed space $C^{\infty}_{\bullet,\bullet}(\overline{D_0})$. Moreover, by \cite{Hamilton79}, the constant in the basic estimates for $\mathfrak{S}^t$ can be chosen to be independent of $t\in \mathbb B$. Thus \eqref{eq:stability} applies. The interested reader is referred to that paper for further information and a clear proof.

\subsection{$L^2$-estimate for $\dbar a=\partial_{\phi} b+c$}\label{ss:dbar} \ \

\medskip

We shall prove a generalization of Demailly's theorem (see \cite{Demailly82}, \cite{Hormander65} or \cite{Bern10}) in this section.

\begin{theorem}\label{th:L2} Let $(L, h)$ be a Hermitian line bundle over an $n$-dimensional complete K\"ahler manifold $(X,\omega)$. Let $v$ be a smooth $\dbar$-closed $L$-valued $(n,1)$-form. Assume that
\begin{equation*}
    i\Theta(L,h) >0 \ \text{on} \ X,\ (resp. \ i\Theta(L,h) \equiv 0 \ \text{on} \ X )
\end{equation*}
and
\begin{equation*}
    I(v):= \inf_{v=\partial_\phi b+c} ||b||^2_{\omega} + ||c||^2_{i\Theta(L,h)} <\infty, \ (resp. \ I(v):= \inf_{v=\partial_\phi b} ||b||^2_{\omega}<\infty ) .
\end{equation*}
Then there exists a smooth $L$-valued $(n,0)$-form $a$ on $X$ such that $\dbar a=v$ and
\begin{equation}\label{eq:L2q}
    ||a||_{\omega}^2 \leq I(v).
\end{equation}
\end{theorem}

\begin{proof} We shall only prove the $i\Theta(L,h) >0$ case, since the $i\Theta(L,h) \equiv 0$ case can be proved by a similar argument. By H\"ormander's theorem and the standard density lemma for complete K\"ahler manifold, it suffices to prove that,
\begin{equation}\label{eq:hormander}
    |(\partial_\phi b+c,g)_\omega|^2\leq (||b||^2_\omega + ||c||^2_{i\Theta(L,h)})(||\dbar^*g||_\omega^2+||\dbar g||^2_\omega),
\end{equation}
for every smooth $L$-valued $(n,1)$-form $g$ with  compact support in $X$. Notice that
\begin{equation*}
    (\partial_\phi b+c,g)_\omega=(b,\partial_\phi^*g)_\omega+(c,g)_\omega.
\end{equation*}
Hence
\begin{equation*}
     |(\partial_\phi b+c,g)_\omega|^2\leq (||b||^2_\omega + ||c||^2_{i\Theta(L,h)})(||\partial_\phi^*g||^2_{\omega}+([i\Theta(L,h),\Lambda_\omega]g,g)_{\omega}),
\end{equation*}
where $\Lambda_\omega$ denotes the adjoint of $\omega\wedge$. Thus \eqref{eq:hormander} follows from the Bochner-Kodaira-Nakano formula. The proof is complete. 
\end{proof}

\end{document}